\input ./style/arxiv-general.cfg
\documentclass[aop,MSNbibl,dvips]{arximspdf}
\makeatletter
   \@ifpackageloaded{graphicx}{}{\usepackage{graphicx}}
\makeatother
\usepackage[mathcal]{euscript}
\usepackage{mathrsfs}
\doi{10.1214/14-AOP934} %kopijuoti is PTS
\volume{43}
\issue{5}
\pubyear{2015}
\firstpage{2332}
\lastpage{2373}
\makeatletter
\newcommand{\VdW}{$(\mbox{\hyperref[pg:Vd]{$VD$}})_\omega$}
\newcommand{\PW}{$(\mbox{\hyperref[pg:Poincare]{$P$}})_\omega$}
\def\mathcal{\mathscr}
\newcommand{\rrvert}{\vert}
\newcommand{\llvert}{\vert}
\def\nicefrac{\frac}
\def\dotsc{\ldots}
\def\dotsb{\cdots}
\newtheorem{theorem}{Theorem}
\newtheorem{corollary}[theorem]{Corollary}
\newtheorem{lemma}[theorem]{Lemma}
\newtheorem{proposition}[theorem]{Proposition}

\newproclaim{definition}[theorem]{Definition}
\newproclaim{assumption}[theorem]{Assumption}
\newtheorem{theoremmm}{Theorem}% \ref{no sublinear harmonic}'
\newproclaim{remark}{Remark}
\newproclaim{example}{Example}[section]
\newproclaim{conjecture}{Conjecture}
\newproclaim{question}{Question}
\newcommand{\Cc}{\mathbf{C}}
\newcommand{\Dd}{\mathbf{d}}
\newcommand{\Z}{\mathbb Z}
\newcommand{\R}{\mathbb R}
\newcommand{\E}{\mathbb E}
\renewcommand{\P}{\mathbb P}
\newcommand{\Bb}{\mathbf{B}}
\newcommand{\Pp}{\mathbf{P}}
\newcommand{\Hh}{\mathbf{H}}
\newcommand{\Ee}{\mathbf{E}}
\newcommand{\law}{\mathcal{L}}
\newcommand{\supp}{\operatorname{supp}}
\newcommand{\eqref}[1]{(\ref{#1})}
\makeatother

\begin{document}
\begin{frontmatter}

\title{Disorder, entropy and harmonic functions}
\runtitle{Disorder, entropy and harmonic functions}

\begin{aug}
\author[A]{\fnms{Itai}~\snm{Benjamini}\ead[label=e1]{itai.benjamini@weizmann.ac.il}},
\author[B]{\fnms{Hugo}~\snm{Duminil-Copin}\ead[label=e2]{hugo.duminil@unige.ch}\thanksref{T1}},
\author[A]{\fnms{Gady}~\snm{Kozma}\corref{}\ead[label=e3]{gady.kozma@weizmann.ac.il}\thanksref{T2}}
\and
\author[C]{\fnms{Ariel}~\snm{Yadin}\ead[label=e4]{yadina@bgu.ac.il}}
\affiliation{Weizmann Institute of Science, Universit\'e de Gen\`eve,\break Weizmann Institute of Science and Ben
Gurion University of the Negev}
\thankstext{T1}{Supported by the EU Marie-Curie RTN CODY, the ERC AG
CONFRA, as well as by the Swiss {FNS} and the Weizmann institute.}
\thankstext{T2}{Supported by the Israel Science Foundation and the
Jesselson Foundation.}
\address[A]{I. Benjamini\\
G. Kozma\\
Department of Mathematics\\
Weizmann Institute of Science\\
Rehovot\\
Israel\\
\printead{e1}\\
\phantom{E-mail:\ }\printead*{e3}}
\address[B]{H. Duminil-Copin\\
D\'epartement de Math\'ematiques\\
Universit\'e de Gen\`eve\\
Gen\`eve\\
Switzerland\\
\printead{e2}}
\address[C]{A. Yadin\\
Ben Gurion University of the Negev\\
Beer Sheva\\
Israel\\
\printead{e4}}
\runauthor{Benjamini, Duminil-Copin, Kozma and Yadin}
%\author{\fnms{}~\snm{}\corref{}}
%\and
%\author{\fnms{}~\snm{}}
%\runauthor{}
%\affiliation{}
%\dedicated{}
%\address{} %adresu isvedimo komanda gale!
%\address{}
\end{aug}

% HISTORY:
\received{\smonth{5} \syear{2013}}
\revised{\smonth{3} \syear{2014}}
%\accepted{\smonth{} \syear{}}

% ABSTRACT
%
\begin{abstract}
We study harmonic functions on random environments with
particular emphasis on the case of the infinite cluster of
supercritical percolation on $\Z^d$. We prove that the vector space
of harmonic functions growing at most linearly is $(d+1)$-dimensional
almost surely. Further, there are no nonconstant sublinear
harmonic functions (thus implying the uniqueness of the
corrector). A main ingredient of the proof is
%the existence of a coupling, given by
a quantitative, annealed version of the Avez
entropy argument. This also provides bounds on the derivative of the
heat kernel, simplifying and generalizing existing results. The
argument applies to many different environments; even
reversibility is not necessary.
\end{abstract}

% KEYWORDS
\begin{keyword}[class=AMS]
\kwd[Primary ]{60K37}
\kwd[; secondary ]{31A05}
\kwd{82B43}
\kwd{37A35}
\kwd{60B15}
\kwd{60J10}
\kwd{20P05}
\end{keyword}

\begin{keyword}
\kwd{Harmonic functions}
\kwd{percolation}
\kwd{random walk in random environment}
\kwd{stationary graphs}
\kwd{entropy}
\kwd{Avez}
\kwd{Kaimanovich--Vershik}
\kwd{corrector}
\kwd{IIC}
\kwd{UIPQ}
\kwd{planar map}
\kwd{anomalous diffusion}
\end{keyword}

%\begin{keyword}[class=AMS]
%\kwd[Primary ]{}
%\kwd{}
%\kwd[; secondary ]{}
%\end{keyword}
%\begin{keyword}
%\kwd{}
%\end{keyword}
\end{frontmatter}

%s1 #&#
\section{Introduction}\label{sec1}

Since the work of Yau in 1975, where the Liouville
property for positive harmonic functions on complete manifolds with
nonnegative Ricci curvature was proved \cite{Yau75}, the structure of various
spaces of harmonic functions has been at the heart of geometric
analysis. Some years later, Yau conjectured that the space of
polynomial growth harmonic functions of fixed order is always finite
dimensional in open manifolds with nonnegative Ricci curvature.
Extensive literature has appeared on this conjecture and related
problems. Understanding progressed quickly (Yau's conjecture was
proved by Colding and Minicozzi \cite{CM97}) and gave birth to many
tools; see \cite{Li93} for an introduction to the subject. %For
%instance, we mention the use of gradient estimates together with
%Harnack's inequalities to control the heat kernel.

In the algebraic setting, bounded harmonic functions played a central role
since the introduction of the Poisson boundary by Furstenberg
\cite{F63,F71}; see also the survey \cite{W94}. Recently, the
geometric approach made a remarkable appearance in the algebraic
realm when Kleiner proved that the space of harmonic functions with fixed
polynomial growth on the Cayley graph of a group with polynomial
volume growth is finite dimensional using the approach of \cite{CM97}.
He used this fact to provide a new proof of Gromov's theorem \cite
{Kle10}; see \cite{ST} for a quantitative version of this theorem.

Another place where harmonic functions have played an important role
recently is in the proof of the central limit theorem on
random graphs. A central element in the proofs (see, e.g., \cite
{SS04,MP07,BB07,GO11}) is
the construction of a harmonic function $h$ on the cluster which is
close to linear---the term $\chi(x)=h(x)-\langle x,v\rangle$ is
called the
corrector and once one shows that $\chi(x)=o(\|x\|)$, the proof may proceed.

The focus of this article is the case of random graphs.
Classical tools of geometric analysis do not extend to
this context in a straightforward way. Indeed, a random environment is
not regular at the microscopic scale. In order to understand harmonic
functions, one thus needs to make use only of the control of the
macroscopic behavior of the environment.
Let us take supercritical percolation as an example; see
\cite{G99} for background and definitions.

For $p\in(0,1)$, consider the random graph $G=(V(G),E(G))$
defined by $V(G)=V(\Z^d)$ and $E(G)$ being a random set containing each
edge of $\Z^d$ with probability $p$, independently of the other edges.
It is classical that (in dimension $d \geq2$)
there exists $p_c(d)\in(0,1)$ such that for $p<p_c(d)$,
there is almost surely no infinite connected component (also called cluster),
while for $p>p_c(d)$, there is a unique infinite cluster. When $p>p_c(d)$,
we denote this cluster by $\omega$.

%th1 #&#
\begin{theorem}\label{thm:percolation}
Let $d\ge2$, and let $p>p_c(d)$. Then with probability 1, the
infinite cluster $\omega$ has no nonconstant sublinear harmonic functions.
\end{theorem}

This immediately shows that the corrector $\chi$ is unique, as was
conjectured by Berger and Biskup \cite{BB07}, Question 3.

In more regular settings, claims of this sort have been proved using
the following strategy: try to show that two random walks starting at
neighbors will couple before time $n$ with\vspace*{1pt} probability bigger than
$1-Cn^{-1/2}$. This fact is classical in the case of the hypercubic
lattice $\Z^d$ where an explicit coupling can be exhibited. In the
random context it is not clear how to construct an explicit coupling,
but a number of approaches in the literature allows one to construct a
coupling indirectly. The known Gaussian heat kernel bounds [see
\eqref{a} below] allow one to construct a coupling that will fail with
probability $n^{-\epsilon}$. Using\vspace*{1pt} also the central limit theorem
already mentioned, one could improve this to $n^{-1/2+o(1)}$. Nevertheless,
getting the precise $n^{-1/2}$ seems difficult with these
approaches. The approach we will apply below not only gives the precise
order $n^{-1/2}$,
but the proof is also significantly simpler than those just suggested.
%'explicit' coupling is much harder to obtain in the random case. and
%we use entropy \cite{Av} to construct an abstract coupling.

%More precisely, we quantify and generalise the celebrated
%Kaimanovich-Vershik
%\cite{KV} argument proving the equivalence on Cayley graphs between
%sublinear entropy for the simple random walk and Liouville's property.
%More precisely, we extend it to the random context (this was already
%done in \cite{BC11}) and we make it quantitative. For stationary
%random graphs of polynomial growth, the argument allows to construct a
%coupling in time between walks starting at neighbors which is close to
%the optimal one.

The proof uses an
entropy argument similar to Avez \cite{Av} who
showed that a Cayley graph satisfies the Liouville property if
the entropy of the random walk on it is sublinear. In fact the ``if'' here
is an ``if and only if'' as was shown by Kaimanovich and Vershik
\cite{KV} and, with a different approach, by Derriennic \cite{D80},
but we will not need the other direction. Two extensions
of this result were known before: it
applies to random graphs \cite{BC10}, and it can be quantified \cite{EK10},
Section~5.
It turns out that the two
generalizations can be applied simultaneously.
Further, Theorem~\ref{thm:percolation} is but an example: the
techniques work in great generality; even reversibility is not needed. Only
stationarity of the walk and some weak (sub-)diffusivity are
used.
%Let us detail the precise assumptions.
Precise assumptions are detailed below.

%pa1.subsection.subsubsection.1 #&#
\subsection*{The environment as viewed from the particle} To state the
full result, we need to define what we mean by ``environment.''
We are interested in environments which are somehow translation invariant.
This notion extends the transitivity condition to the random
context. Historically, this traces to the works of Papanicolaou and
Varadhan \cite{PV82} and Kozlov \cite{K85} who studied random walk in
random environments on
$\mathbb{Z}^d$ by translating the environment so that the walker
remains at $\vec{0}$. In other words, instead of having a walker
move around in some environment, the walker stays at the origin, and the
environment moves ``below'' it, hence the name \emph{the environment
as viewed from the particle}. When the distribution of the
environment stays the same after a single step of this process, the
environment is called \emph{stationary}.

The notion was extended beyond $\mathbb{Z}^d$ in \cite{LPP95} who
showed a similar phenomenon for Galton--Watson trees: when you do a
single step of random walk starting from the root of the tree, the
resulting random graph has the same distribution with respect to the new
position of the walker.

In such examples the most natural definition of ``having the same
distribution'' uses isomorphisms (in \cite{LPP95} this could be
avoided due to the very simple structure of trees, but it appears,
e.g., in \cite{AL07,BC10}). The resulting definition looks a little
abstract at first, but in fact is very easy to verify in examples. For
example, in the $\mathbb{Z}^d$ case, the isomorphisms would be
translations, while in the Galton--Watson case, they would be a change
of root followed by an arbitrary map. Let us give the details.

Consider a Markov chain $(X_n)_{n\ge0}$ taking values in some set $V$.
The law of this chain can be encoded by a function
$P\dvtx V\times V\to[0,1]$ where $P(x,y)$ denotes the probability to move
from $x$ to $y$. We always assume that our
Markov chain is irreducible, that is, that for any $v,w\in V$ there is
an $n$ such that $P^n(v,w)>0$. A~\emph{rooted Markov chain} is a
triplet $(P,V,\rho)$ where $\rho\in V$ is some vertex that will be
called the \emph{root vertex}. Two rooted Markov chains $(P,V,\rho)$
and $(P',V',\rho')$
are considered
isomorphic if there is a one-to-one map $\phi\dvtx V\to V'$ with $\phi
(\rho
)=\rho'$ and
$P(x,y)=P'(\phi(x),\phi(y))$.

We define an \emph{environment as viewed from the
particle}, abbreviated as simply \emph{environment}, to be a
random rooted Markov chain.
Two environments are considered to have
the same law if they are identical as measures on isomorphism classes
of rooted Markov chains (alternatively, if they
can be coupled in such a way that the resulting rooted
Markov chains are isomorphic with probability 1).

%de2 #&#
\begin{definition}
An environment $(P,V,\rho)$ is called \emph{stationary} if it has the same
law as $(P,V,X_1)$ where $X_1$
is sampled from $P(\rho,\cdot)$.
% -------------------------------> changed here
%is the position of the first step of
%the Markov chain $P$, starting from $\rho$, and independent of the
%random choice of $(P,V,\rho)$.
\end{definition}

%Do you think something like that would help? Frankly after writing
%this I decdided I do not like it.
%Formally, let $\Omega$ is a probability space such that $(P, V, \rho)$
%is a measurable function from $\Omega$ into the space of triplets, and
%such that $X_1$ is a measurable function from $\Omega$ into $V$ such
%that
%$\P(X_1=v | P,V,\rho)=P(\rho,v)$ almost surely. Then
%$(P,V,X_1)$ is a new measurable function from $\Omega$ into the space
%of triplets, and hence is an environment. We require that it has the
%same law as $(P,V,\rho)$.

As we already remarked, stationary environments are very common, and we
provide ten examples in the
end of Section~\ref{sec:entropy}. %Note that the definition above
%requires that $(P,V,\rho)$ and $(P,V,X_1)$ are constructed on the same
%measurable space $(\Omega,\mathcal F)$ in order to compare their laws.
%The construction of $(\Omega,\mathcal F,\mathbb P)$ is not always
%straightforward (sometimes one works directly with equivalence classes
%of rooted Markov chains, where two Markov chains are isomorphic if
%there is a one-to-one map $\phi:V\to V'$ with $\phi(\rho)=\rho'$ and
%$P(x,y)=P'(\phi(x),\phi(y))$).
Most of these examples are embedded
in $\Z^d$, and for these we could have used the definition of
\cite{PV82,K85}. % i.e.\ define an environment to be simply
%and for these we could have used a much simpler definition
%--- there is no need to ``couple the environments such that the
%resulting
%rooted Markov chains are isomorphic with probability 1''.
%and a random rooted Markov chain can be considered as
%a random
%function $P:\Z^d\times\Z^d\to[0,1]$, after shifting $\rho$ to
%$\vec{0}$ (formally this is a stronger requirement, but all our
%examples
%embeddable into $\Z^d$ have it).
Examples~\ref{exmp:poisson}, \ref{exmp:GW} and \ref{exmp:UIPQ},
however, are not embeddable into $\Z^d$, so the isomorphism cannot be
taken to be a ``translation,'' though constructing it is still easy.

A very important subset of stationary environments is given by
environments $V$ with the
structure of a weighted graph [with the weight being a symmetric
positive function $\nu$
on every edge $(v,w)\in E$, and $0$ on every pair $(v,w)\notin
E$]. In such case, $P$ is given by
\[
P(v,w)=\frac{\nu(v,w)}{\nu(v)} \qquad\mbox{where } \nu(v) = \sum
_x \nu (v,x).
\]
These environments will be called \emph{random stationary graphs}. This
particular type of
Markov chain is also commonly called \emph{reversible}. The reversible
case has a rich theory; see, for example, \cite{AL07,BC10} where one
can also find many more
examples. To clearly distinguish between the reversible and nonreversible
case, random stationary graphs will be denoted by $(G,\nu,\rho)$
where $G$ is the graph, $\nu$ is the weight function and $\rho$
is the root.
%One may consider
%for example reversible RWRE (which are just random subgraphs of $\Z^d$
%invariant to translations), and then the problem of stationarity is
%solved. One needs only reweight
%by the degree at 0 to become stationary in our sense. Another set of
%examples comes from taking a sequence of finite graphs, picking the
%root uniformly, and taking the Benjamini-Schramm limit \cite{BS01}.

The graph distance in $G$ is denoted by $\Dd^G(\cdot,\cdot)$ and the
ball of size $r$ centered at $x$ by $\Bb^G_x(r)$. We will also
consider this distance in nonreversible setting, where it is simply the
smallest $n$ such that $P^n(x,y)>0$
(in this case it may fail to be a metric).
Since the distinction between
annealed and quenched statements will be clear in the context, we will
often drop the dependence on $G$ in the notation. For instance,
$\Pp_x^G$, $\Dd^G(\cdot,\cdot)$ and $\Bb^G_x(n)$ will become simply
$\Pp_x$, $\Dd(\cdot,\cdot)$ and $\Bb_x(n)$. For the convenience of the
reader, we collected the notation
and conventions used in this paper in the last section of the
introduction (page \pageref{pg:notation}).

%pa1.subsection.subsubsection.2 #&#
\subsection*{Nonconstant harmonic functions with minimal growth} Let
$P$ be a Markov chain with state space $V$. Then a function $h\dvtx V\to
\R$ is called harmonic if $h(X_n)$ is a martingale, or in other
words, if
\[
h(x)=\sum_yP(x,y)h(y) \qquad\forall x.
\]
%
%The
%space of nonconstant harmonic functions with minimal growth plays a
%specific role in the study of the underlying space.

As already mentioned, harmonic functions have had a number of
important applications recently. Let us expand on the particular
application in Kleiner's proof of
Gromov's theorem \cite{Kle10}. It was known since the 1970s that in order
to prove Gromov's
theorem, it is enough to show that any group with polynomial volume
growth has
a nontrivial finite-dimensional representation. Kleiner showed that
any group has a nontrivial linearly growing harmonic function, and that on
groups with polynomial growth, the dimension of the space polynomially
growing harmonic
functions is finite. Since
the group acts on harmonic functions on its Cayley graph by
translations, this provides a finite dimensional representation and
proves Gromov's theorem. Shalom and Tao
\cite{ST} showed that a quantitative version of Kleiner's proof can be
performed. Further, they characterized the linearly growing harmonic
functions (for groups with polynomial volume growth these are the
nonconstant harmonic functions with minimal growth \cite{HSC93}, Theorem~6.1). They showed (personal communication) that when the group is
nilpotent, any such function must be a character of the group
(or the sum of a character and a constant), in
analogy to the Choquet--Deny theorem \cite{CD60,M66}. For virtually
nilpotent groups this holds \emph{mutatis mutandis}. We plan to
analyze harmonic functions with minimal growth in the context of
Cayley graphs, especially of wreath products, in a future paper. %The
%main contribution of this paper is the study of minimal growth
%harmonic functions on stationary random graphs. %For random
%environments, harmonic functions of minimal growth have also be proved
%to be useful (for instance on infinite percolation clusters, where
%they correspond to projections of the so-called corrector).

We now return to the setting of this paper, that is, of stationary random
graphs. Using the entropy of the random walk, it is possible to bound from
below the minimal growth of nonconstant harmonic functions in terms
of the rate of escape of the random walk. A particularly interesting
case is provided by stationary environments with diffusive behavior,
for which the bound is often sharp. A stationary environment
$(P,V,\rho)$ satisfies \emph{diffusive or subdiffusive behavior}
$(\mathit{DB})$ if
%\newcounter{phoney}
%\refstepcounter{phoney} %makes pg:SDB point here
%
{\renewcommand{\theequation}{DB}
\begin{equation}\label{pg:SDB}
%(DB)
\mbox{there exists $C>0$ such that $\mathbb E \bigl(\Dd(
\rho,X_n)^2 \bigr)\leq Cn$ for every $n$.}
\end{equation}}
\hspace*{-3pt}Here and below $\E$ is the average over both the environment and
over the walk (the so-called annealed average). We may now state our
main result.

%\newcommand{\SDB}{}
%\newcommand{\SDB}{\hyperref[pg:SDB]{$(DB)$}}

%th3 #&#
\begin{theorem}\label{no sublinear harmonic}
Let $(P,V,\rho)$ be a stationary environment such that\break $\mathbb
E (|\Bb_\rho(n)| )\leq Cn^d$ for some constants
$C,d<\infty$
independent of $n$. If $(P,V,\rho)$ satisfies (\ref{pg:SDB}), then for almost
every environment, there are no nonconstant sublinear harmonic functions.
\end{theorem}

We say that $h$ is a sublinear function if $h(x)=o(\Dd(\rho,x))$ as
$\Dd(\rho,x)\to\infty$. Restricting to the case of percolation, it is
also quite natural to ask what happens with functions which are
sublinear with respect to the Euclidean distance $\|x\|$ (e.g.,
this is how the question is formulated in \cite{BB07}). The
result of Antal and Pisztora \cite{AP96} yields that graph and
Euclidean distances are comparable on the infinite cluster, and that
therefore the previous question follows from Theorem~\ref{no sublinear
harmonic}.

As already stated, Theorem~\ref{no sublinear harmonic} applies to many
different models, some of them
significantly less well understood than percolation. See a list of
examples at the end of Section~\ref{sec:entropy}.

%This theorem should be compared to the Cayley graph case. In this
%context, it is known that groups of polynomial growth do not admit
%sublinear growth harmonic functions.

%\newcommand{\thmtp}{}

%\newcommand{\thmtp}{\hyperref[page:2p]{\ref*{no sublinear harmonic}'} }

Whether (\ref{pg:SDB}) follows from
polynomial growth in the reversible case is an interesting question.
The Carne--Varopoulos bound
\cite{Car85,Var85} gives that $\Ee_\rho( d(\rho,X_n))\le C\sqrt
{n\log
n}$, which would give (with the same proof as that of Theorem~\ref{no
sublinear harmonic}; see Theorem \ref{page:2p} in Section~\ref{sec:entropy})
that any stationary
random graph with polynomial volume growth has no nonconstant harmonic
functions
$h$ with $h(x)\le C\Dd(\rho,x)/\sqrt{\log\Dd(\rho,x)}$. Without
stationarity the
Carne--Varopoulos bound $\sqrt{n\log n}$ cannot be improved, as was
shown by Barlow and Perkins \cite{BP89}. Kesten gave a beautiful
argument that a
stationary random graph embedded in $\Z^d$ satisfies (\ref{pg:SDB}); see, for example,
\cite{BP89}, Section~2. But it does not seem to apply just assuming
polynomial growth.

%Let us mention that \SDB{} is not very strong. For instance, it is
%true in most of the environments with polynomial growth. It is
%reminiscent of the Varopoulos-Carne bound for Cayley graphs with
%polynomial growth \cite{Car85,Var85}.

The relation between entropy, harmonic functions and speed of the
random walk
holds for more general environments
(e.g., with larger growth).
We defer to Section~\ref{sec:entropy} for a more complete account of
this question.

%pa1.subsection.subsubsection.3 #&#
\textit{Polynomially growing functions}. As
in the case of manifolds, we are interested in the dimension of
the space of
harmonic functions with prescribed polynomial growth. Of course, one
can encounter very different behavior depending on the environment
(like in the deterministic case). Hence we will assume that our
environments satisfy volume doubling and
the Poincar\'e inequality.
%In this section we do not assume stationarity,
%and hence we do not need a root for our graph. The object of interest
%is therefore simply a random weighted graph $(G,\nu)$.
Here is the
precise formulation of our assumptions on the environment:
let $(G,\nu,\rho)$ be a rooted weighted graph.

\newcommand{\CVD}{\Cc_{\mathrm{VD}}}
\newcommand{\CP}{\Cc_{\mathrm{P}}}

%\medbreak
%\refstepcounter{phoney}
%
\begin{longlist}[$(\mathit{VD})_G$.]
%\begin{tabular}{c p{0.7\textwidth} }
%\label{pg:Vd}
\item[$(\mathit{VD})_G$.] $(G,\nu,\rho)$ satisfies the \emph{anchored volume doubling property}
$(\mathit{VD})_G$ if there exists
$0<\CVD<\infty$ such that the following holds. For every $\lambda
<\infty$,
there exists $n_0\in\mathbb N$ such that for all $n>n_0$,
and for every $x\in\Bb_\rho(\lambda n)$,
\[
\nu\bigl(\Bb_x(2n)\bigr) \le \CVD\nu\bigl(\Bb_x(n)
\bigr), %
\]
where $\nu(\Bb)$ is the total weight of the edges in the ball $\Bb$.
%%ball of
%radius $r$ centered at $x$.
%
%
%\begin{tabular}{c p{0.7\textwidth} }
%\refstepcounter{phoney}
%\label{pg:Poincare}
\item[$(P)_G$.]
$(G,\nu,\rho)$ satisfies the \emph{anchored Poincar\'e inequality}
$(P)_G$ if there exists $\CP<\infty$
such that the following holds. For every $\lambda<\infty$,
there exists $n_0\in\mathbb N$ such that for all $n>n_0$,
for every $x\in\Bb_\rho(\lambda n)$ and every $f\dvtx \Bb_x(2
n)\rightarrow
\mathbb R$,
\[
\sum_{y\in\Bb_x(n)} \bigl(f(y)-\overline{f}_{\Bb_x(n)}
\bigr)^2\nu (y)\leq \CP n^2\sum
_{(y,z)\in E(\Bb_x(2n))}\bigl|f(y)-f(z)\bigr|^2\nu(y,z),
\]
where
\[
\overline{f}_{\Bb_x(n)}=\frac{1}{\nu(\Bb_x(n))}\sum_{y\in\Bb
_x(n)}f(y)
\nu(y). %
\]
\end{longlist}

%% $\CVD$ and $\CP$ may depend on the random choice of $(G,\nu)$ (though
%% we do not have any interesting example which actually uses this
%% freedom). The minimal $n$ from which the properties hold may depend
%% both on the environment and on $\lambda$.

%\newcommand{\VdG}{}
%\newcommand{\PG}{}
%%For percolation
%\newcommand{\VdW}{}
%\newcommand{\PW}{}

%\newcommand{\VdG}{$(\mbox{\hyperref[pg:Vd]{$VD$}} )_G$}
%\newcommand{\PG}{$ (\mbox{\hyperref[pg:Poincare]{$P$}} )_G$}
%For percolation
%\newcommand{\VdW}{$ (\mbox{\hyperref[pg:Vd]{$VD$}} )_\omega$}
%\newcommand{\PW}{$ (\mbox{\hyperref[pg:Poincare]{$P$}} )_\omega$}

%(if you are reading the online version, \VdG{} and \PG{} will usually
%be hyperlinks here, click them if you forget the definition).

Similar properties are classical in geometric analysis. They
go back to
the theory developed by De Giorgi, Nash and Moser \cite
{Mos61,Mos64,Nas58,DeG57} in the fifties and sixties for uniformly
elliptic second-order
operators in divergence form. In the classic context, they imply
the Harnack principle and Gaussian bounds for the heat kernel. While
the definitions above have no randomness in them, they are tailored
for the random case: they take into
consideration that in most examples of interest these properties do not
hold from every point since some unusual points always exist. For this
reason, the properties are required to hold for balls which are not too
far from our root $\rho$,
relative to their size. This is reminiscent of Barlow's good and very
good balls \cite{Bar04}, but our requirements are much weaker, we only
need the properties to hold for ``macroscopic balls,'' balls whose
distance to $\rho$ is proportional to their radius.

Let us remark on the appearance of the number 2 in $\Bb_x(2n)$ in both
properties. For the volume doubling property it is clear that these
properties are equivalent for all choices bigger than 1; that is, if
one was to
define a ``3-volume doubling property,'' then it would be equivalent to
the ``2-volume doubling property'' defined above, though perhaps with
different $\CVD$ and minimal $n$. The same holds for
the Poincar\'e inequality, under the assumption of volume
doubling. This is well known in the standard
settings (see, e.g.,  \cite{J86}, Section~5), and the proof
carries over to
the anchored case without any change.

With these definitions we can state the following easy but, we
believe, conceptually important theorem. Note that the theorem is for
a fixed graph (though the most interesting applications are for random graphs).

%th4 #&#
\begin{theorem}\label{finite dimension}
Let $(G,\nu,\rho)$ be a rooted weighted graph. If $(G,\nu,\rho)$
satisfies $(\mathit{VD})_G$ and $(P)_G$, then for every
$k>0$, the space of harmonic functions with $|h(x)| \le C\Dd(\rho,x)^k$
for all $x$ far enough from $\rho$, is finite dimensional.

Further, the bound on the dimension depends only on $k$, $\CVD$ and
$\CP$,
and not on $n_0(\lambda)$.
\end{theorem}

This theorem represents a discrete anchored version of Yau's conjecture
except that the Poincar\'e inequality must be assumed since it is not
automatically satisfied (in Yau's settings every manifold with
nonnegative Ricci curvature satisfies a Poincar\'e inequality
\cite{B82} while in Kleiner's, every group satisfies an appropriate
version of the Poincar\'e inequality; see, e.g.,  \cite{PSC}, Lemma~4.1.1).
%Also, in those settings, the Poincar\'e inequalities are similar to
%ours,
%but since there is no disorder and there is no need to consider
%unusual points.
The proof of this theorem follows the existing strategy developed in
\cite{CM97,Del98,Kle10,ST,T10}. Let us stress again that the
interesting part
is that it requires only \emph{macroscopic} volume growth and
Poincar\'e inequality: the definitions of $(\mathit{VD})_G$ and $(P)_G$ only
examine balls of radius $n$ inside $\Bb_\rho(\lambda n)$ for some finite
$\lambda$.

%% It is natural to ask if Theorem~\ref{finite dimension} can be
%% strengthened by replacing the condition of anchored volume doubling
%% property with a (suitably defined) \emph{anchored
%% polynomial growth property}. This, however, fails with the most
%% natural definition of such a property due to the example of the
%% comb. The comb is a subgraph of $\Z^2$ containing all vertices, all
%% vertical edges (the ``teeth'' of the comb) and the horizontal edges
%% along the axis (the ``frame'' of the comb). It is not difficult to
%see
%% that this graph satisfies the Poincar\'e inequality \PG{} and any
%% reasonable definition of polynomial growth, but has an infinite
%% dimensional space of Lipschitz harmonic functions (we will not prove
%% any of these claims as this will take us too far off-topic).

%While in the context of Cayley graphs, this is equivalent
%to having volume growth and Poincar\'e inequality for any
%scale, it is no longer the case in the random context. In
%fact, uniform Poincar\'e inequality and volume growth do
%not always hold in the random case. In order to understand
%this fact, recall that most of the random environments
%will contain arbitrary graphs microscopically. Hence,
%conditions \VdG{} and \PG{} are natural relaxations of the
%deterministic conditions, and they are satisfied by many
%random environments.

When we apply Theorem~\ref{finite dimension}, the graph $G$ will be
random. Since the dimension depends only on $\CVD$ and $\CP$, then in
particular, if these constants are not random, neither is the
bound. Thus, for example, in supercritical percolation there is
a constant $A$ (depending only on the dimension $d$ and the probability
$p$) such that $\CVD\le A$ and $\CP\le A$ almost surely (the minimal
$n$ is the only quantity which really changes between
configurations). Hence for each $k$ there is a number $D_k$ such that the
dimension of harmonic functions of growth at most of order $\Dd(\rho,x)^k$ is
smaller than $D_k$, almost surely. We discuss a few other examples of
random graphs satisfying $\CVD$ and $\CP$ in the end of Section~\ref
{sec:polynomial growth}, but in general one should keep in mind that
the Poincar\'e inequality restricts the behavior of random walk on the
graph significantly, so Theorem~\ref{finite dimension} applies in much
less generality than Theorem~\ref{no sublinear harmonic}.

%pa1.subsection.subsubsection.4 #&#
\subsection*{Linearly growing functions} In the special case of
environments which are modifications of $\Z^d$, we can compare the
dimension of harmonic functions with a prescribed growth to the
dimension of harmonic functions on $\Z^d$.
The simplest perturbation of $\Z^d$ is the supercritical cluster of
percolation.
We prove the following theorem.

%For $p\in(0,1)$, consider the random graph $G=(V(G),E(G))$ defined by
%$V(G)=V(\Z^d)$ and $E(G)$ being a random set
%containing each edge of $\Z^d$ with probability $p$, independently of
%the other edges.
%It is classical that there exists $p_c(d)\in(0,1)$ such that for
%$p<p_c(d)$,
%there is almost surely
%no infinite connected component (also called cluster), while for
%$p>p_c(d)$,
%there is a unique infinite cluster. When $p>p_c(d)$, we denote this
%cluster by $\omega$.

%th5 #&#
\begin{theorem}\label{dimension linear}
Let $d \geq2$. For $p>p_c(d)$, let $\omega$ be the unique infinite component
of percolation on $\Z^d$.
Then,
the dimension of the vector space of harmonic
functions with growth at most linear on $\omega$ is equal to $d+1$
almost surely.
\end{theorem}

This theorem must be understood as a first step toward a bigger goal,
which would be to compute the dimension of all spaces of harmonic
functions with prescribed (polynomial) growth.

The properties of the supercritical percolation cluster used in this
proof are quite general: the
$d$-dimensional volume growth and the Poincar\'e inequality
$(P)_{\omega}$
proved (in stronger form) by Barlow \cite{B04} as well as the
Gaussian bounds which Barlow concludes from these, and an invariance
principle \cite{SS04,BB07,MP07}. All these properties witness the close relation between
macroscopic properties of the supercritical percolation cluster and
$\mathbb R^d$. In some sense, it confirms the heuristic that this cluster
is an approximation of $\Z^d$. %Other random environments share the
%same properties as the supercritical cluster. For those, a result
%similar to Theorem~\ref{dimension linear} can be derived. In Section~\ref{sec:linear growth}, we give several examples of such environments.

%A byproduct of the fact that the infinite cluster of percolation can
%be seen as a stationary random graph is the fact that the so-called
%\emph{corrector} is unique. In \cite{BB07,MP07}, it is shown that for
%every $d\ge2$, there exists $\chi:\mathbb Z^d\times\Omega\rightarrow
%\mathbb R^d$ such that $x\mapsto x+\chi(x,\omega)$ is harmonic on $
%\omega$, and
%\begin{equation}\label{1}\lim_{n\rightarrow\infty} \frac1n\sup_{x\in
%\Bb_\rho(n)}\big|\chi(x)\big| = 0 a.s.\end{equation}
%This function is called the corrector (actually, \eqref{1} was only
%proved in dimension 2, the general statement was derived in
%\cite{BS10}). In \cite{BB07}, it was asked if the corrector was
%unique. We answer affirmatively this question:

%\begin{theorem}\label{unique corrector}
%For almost every environment $\omega$, there exists only one
%vector-valued function $x\mapsto\chi(x,\omega)$ on $\omega$ such that
%$x\mapsto x+\chi(x,\omega)$ is harmonic on $\omega$,
%$\chi(\rho,\omega)=0$ and $\chi(x,\omega)/\Dd(\rho,x)\rightarrow0$ as
%$\Dd(\rho,x)\rightarrow\infty$.\end{theorem}
%\begin{pf}
%The difference between two such functions $\chi_1$ and $\chi_2$ is a
%sublinear harmonic function on $\omega$. Therefore, by theorem
%\ref{thm:percolation} it must be constant. Moreover, it equals 0 at 0,
%implying $\chi_1=\chi_2$.
%\end{pf}

%pa1.subsection.subsubsection.5 #&#
\subsection*{Heat kernel estimates}
Classically \cite{DeG57,Mos61,Mos64,Nas58}, the kernels of symmetric
diffusions are known to have some H\"older
regularity. In random environments, few results
are known on H\"older behavior: Conlon and Naddaf \cite{CN00} and
Delmotte and Deuschel \cite{DD05} treated the case of random
conductance with a uniform ellipticity condition; see also
\cite{GO11}. The entropy techniques developed for the proof of Theorem~\ref{thm:percolation} allow one to give a very short proof that the
space derivative exists. Moreover, it applies in a very general
context. We present the case of percolation.

%th6 #&#
\begin{theorem}\label{heat kernel}
Let $d \geq2$ and $p>p_c(d)$.
Let $\mathbb P_p$ be the measure of the infinite cluster of
percolation (denoted $\omega$) on $\Z^d$. There exist $C_3,C_4>0$ such
that for every $n>0$ and $x,x',y$ at distance less than $n$ of 0,
if $x$ and $x'$ are adjacent,
\begin{eqnarray*}
&&\mathbb E_p \bigl[ \bigl(\mathbf{p}_n(x,y)-
\mathbf{p}_{n-1}\bigl(x',y\bigr) \bigr)^2 {
\mathbf1}_{ \{y \in\omega\}}{\mathbf1}_{\{
x\ \mathrm{and}\ x'\ \mathrm{are\ adjacent\ in\ } \omega\} } \bigr]\\
&&\qquad\le\frac{C_3}{n^{d+1}}\exp
\bigl[-C_4\Dd(x,y)^2/n\bigr],
\end{eqnarray*}
where $\mathbf{p}_n(y,x):= \Pp_y(X_n=x)$ and $X_n$ is the random
walk on
$\omega$.
\end{theorem}

Estimates for the heat kernel itself (i.e., not for the derivative)
are well understood, and are known as Gaussian estimates
$(\mathit{GE})$. Heuristically, Gaussian estimates are bounds of the form
\[
\frac{C_1}{n^{d/2}}\exp \bigl[-C_2\Dd(x,y)^2/n \bigr] \le
\Pp_x[X_n=y] \le\frac{C_3}{n^{d/2}}\exp
\bigl[-C_4\Dd(x,y)^2/n \bigr].
\]
A few caveats are in place, though. The lower bound cannot hold if
there is any kind of periodicity (as in $\Z^d$ or in subgraphs of it,
such as supercritical percolation). One should talk about continuous
time random walk, lazy random walk, or replace $\Pp_x[X_n=y]$ with
$\Pp_x[X_n=y]+\Pp_x[X_{n+1}=y]$. Further, the lower bound does not
hold for $x$ and $y$ extremely far away---if $\Dd(x,y)>n$, then the
probability is just zero (in the simple random walk case).

In the case of the infinite cluster of supercritical percolation,
these bounds were obtained for continuous time random walk in
\cite{Bar04}. They also hold for simple random walk, most of the details
are filled in \cite{BH09}. Again, one should be careful, as (with
small probability) the environment in the neighborhood of $\rho$ might
be atypical, breaking these estimates for small $n$. Hence the
formulation is as follows.
There exist strictly positive constants $C_1$, $C_2$,
$C_3$ and $C_4$ such that for %every $\delta>0$,
almost every
environment $\omega$ there exist random variables $n_x(\omega), x \in
\Z^d$
so that
for every $x,y\in\omega$ and $n>\max\{n_x(\omega),\Dd(x,y)\}$
%
%e1 #&#
\setcounter{equation}{0}
\begin{eqnarray}
\label{a} \frac{C_1}{n^{d/2}}\exp \bigl[-C_2\Dd(x,y)^2/n
\bigr] &\le& \Pp_x[X_n=y] +\Pp_x[X_{n+1}=y]
\nonumber
\\[-8pt]
\\[-8pt]
\nonumber
&\le&\frac{C_3}{n^{d/2}}\exp \bigl[-C_4\Dd(x,y)^2/n \bigr].
\end{eqnarray}
%
% Here $(\tilde X_n)$ denotes the \emph{lazy} random walk on $\omega$,
% which stays at a site with probability 1/2 and moves with
% probability 1/2. The upper bound in the previous inequalities is
% also true for the nonlazy random walk.
Moreover, the
random variables $n_x(\omega)$ satisfy a stretched exponential
estimate, that is,
%
%e2 #&#
\begin{equation}
\label{eq:stretched} \mathbb P_p\bigl(x \in\omega, n_x(
\omega)\ge s\bigr)\le c e^{-c s^\varepsilon}
\end{equation}
for some $\varepsilon>0$. %Note that the need to use the random
%variable $n_0$ is due to the fact that Gaussian bounds might fail
%microscopically (think of the fact that any finite subgraph of $\Z^d$
%can be found in the percolation configuration).

For the proof of Theorem~\ref{heat kernel} we only need the upper
bound in \eqref{a}. For the proof of Theorem~\ref{dimension linear} we
will also need the lower bound, but only in the regime $|x-y|\approx
\sqrt{n}$, that is, in the regime where the probabilities are of order
$n^{-d/2}$.

%pa1.subsection.subsubsection.6 #&#
\subsection*{Organization of the paper}
In the next section, we study the notion of mean entropy of random
walks on a stationary random graph to bound the total variation
between random walks starting at neighbors. We deduce Theorem~\ref{no
sublinear harmonic}. Section~\ref{sec:polynomial growth} contains
the proof that $(\mathit{VD})_G$ and $(P)_G$ imply that the space of harmonic functions
of prescribed polynomial growth is finite dimensional, that is,
 Theorem~\ref{finite dimension}.
Section~\ref{sec:linear growth} deals with the example of the supercritical
percolation cluster and analyzes the space of linearly growing
harmonic functions. It is completely independent of Section~\ref
{sec:polynomial growth}. Section~\ref{sec:heat kernel} contains the
proof of Theorem~\ref{heat kernel}. Section~\ref{sec:open question}
regroups some open questions.

%pa1.subsection.subsubsection.7 #&#
\subsection*{Notation}\label{pg:notation} To make the distinction
between the reversible and nonreversible case clear, we call the
general case ``Markov chain'' and denote it by $(P,V)$, where $V$ is the
space and $P\dvtx V\times V\to[0,1]$ are the transition probabilities,
$P(x,y)$ being the probability to move from $x$ to $y$. We often
write $P^n$ which we interpret as a matrix power---of course,
$P^n(x,y)$ is also the probability that a random walk starting from
$x$ will be at $y$ after $n$ steps.

Any reversible chain can be described as a random walk on a weighted
graph. If $G$ is a graph and $\nu$ is a function on the edges of $G$ taking
values in $[0,\infty)$, then the Markov chain is given by
$P(x,y)=\nu(x,y)/\sum_z\nu(x,z)$. Here and below, $\nu(x,y)$ for two
vertices $x$ and
$y$ is the weight of the edge $(x,y)$. In particular,
$\nu(x,y)=\nu(y,x)$, and if $(x,y)$ is not an edge of the graph,
then we set $\nu(x,y)=0$. We will always denote reversible Markov
chains by $(G,\nu)$.
%A special case is the reversible case, usually called a ``weighted
%graph''
%and denoted by $(G,\nu)$, where
%$G$ is a graph and $\nu$ are weights on the edges
%\emph{i.e.} $\nu:E(G)\to[0,\infty)$.
%$\nu$ is a symmetric function on supported on the edges of $G$.
We denote by $E(G)$ the
set of edges of the graph $G$, and for a set of vertices $S$ we
denote by $E(S)$ the set of edges between the vertices of $S$. The
notation $x\sim y$ for two vertices will mean that $(x,y)\in E(G)$,
that is, that they are
neighbors in the graph.

We also consider $\nu$ as a
measure. For a vertex $x$, we will denote $\nu(x)=\sum_{y\sim x}\nu(x,y)$
while for a set of vertices $S$, we will denote $\nu(S)=\sum_{x\in
S}\nu(x)$.
Note that edges between two vertices of $S$ are counted
twice in this sum.

% The corresponding Markov chain in this case is given by the
%transition probabilities
%$P(x,y) = \frac{ \nu(x,y) }{\nu(x) }$ where $y \sim x$.

For a fixed graph or Markov chain we denote
by $\Ee$ the expectation with respect to the random walk on that fixed
graph. When the starting point of the random walk is specified, we
will use subscripts and write for instance $\Ee_\rho$. The symbol $\E$
is used to denote the expectation with
respect to both the environment and the random walk (the ``annealed''
average). Similarly, bold letters will usually denote ``quenched''
objects, that is, objects related to an instance $G$ of the
environment. The quantity $\Dd(x,y)$
will denote the graphical distance between two vertices $x$ and $y$ of
$G$, that is, the length of the shortest path in $G$ between $x$ and $y$,
or, in the nonreversible setting, the minimal $n$ such that
$P^n(x,y)>0$. The ball
$\{y\dvtx \Dd(x,y)\le r\}$ will be denoted by $\Bb_x(r)$.

Constants which depend on the
environments $G$ are denoted $\mathbf{c}_i$, while constants of the form
$C_i$ will refer to constants uniform in
the environment. We will occasionally write $\mathbf{c}$ or $C$ for a
constant---different appearances of $\mathbf{c}$ or $C$ might be different
constants.

The cardinality of a set $E$ will be denoted by $|E|$.

%s2 #&#
\section{The entropy argument}\label{sec:entropy}

The connection between entropy and random walks was first exhibited by Avez
\cite{Av} and then made famous in a celebrated paper of Kaimanovich
and Vershik \cite{KV}; see also Derriennic \cite{D80}. For any
discrete variable
$X$ the entropy is defined by
\[
H(X)=\sum_x\phi\bigl(P(X=x)\bigr)\qquad \mbox{where }
\phi(0)=0\mbox{ and }\phi (t)=-t\log t\mbox{ for any }t>0.
\]
Conditional entropy can be defined by
\[
H(X|Y) = \E\bigl[ H(X | Y=y) \bigr] = \sum_y
P(Y=y) \sum_x \phi\bigl( P(X=x | Y=y) \bigr).
\]
It is then quite simple to show that $ H(X|Y) = H(X,Y) - H(Y) $
and
that
$H(X|Y,Z) \leq H(X|Y)$ for any three random variables $X$, $Y$ and $Z$.

Consider a stationary environment $(P,V,\rho)$ with law $\mathbb P$.
Conditionally on $(P,V,\rho)$, define the \emph{entropy} of the random
walk at times $n,m$
started at $\rho$ by
\[
\Hh_{n,m}(P,V,\rho)=H(X_n,X_m)=\sum
_{x,y\in V}\phi \bigl(\Pp_\rho (X_n=x,X_m=y)
\bigr).
\]
When $n=m$, we simply denote $\Hh_{n,n}(P,V,\rho)$ by $\Hh
_n(P,V,\rho
)$. In the random context, we define the \emph{mean entropy} (see
\cite
{BC10}) by
\[
H_{n,m}=\mathbb E\bigl[\Hh_{n,m}(P,V,\rho)\bigr] \quad\mbox{and}\quad
H_n=\mathbb E\bigl[\Hh _n(P,V,\rho)\bigr].
\]

There are many ways of measuring the distance between two probability
measures $\mu$ and $\nu$ on some set $V$, the most standard one being
the total variation
\[
\Vert \mu-\nu\Vert _{\mathrm{TV}}:= \frac{1}2\sum
_{x\in V}\bigl|\mu(x)-\nu(x)\bigr|.
\]
In this article, we will use a less standard one. Define $\Delta(\mu,\nu
)$ by the formula
%
%e3 #&#
\begin{equation}
\label{eq:defDelta}\Delta(\mu,\nu):= \biggl[\sum_{x\in
V}
\frac{ (\mu(x)-\nu(x) )^2}{\mu(x)+\nu(x)} \biggr]^{1/2}.
\end{equation}
Estimating the distance using $\Delta$ is stronger than via the total
variation: by Cauchy--Schwarz,
%
%e4 #&#
\begin{eqnarray}\label{eq:TVDelta}
2 \|\mu-\nu\|_{\mathrm{TV}}&=&\sum_{x\in V}\bigl|\mu(x)-
\nu(x)\bigr|=\sum_{x\in
V}\sqrt{\mu (x)+\nu(x)}
\frac{|\mu(x)-\nu(x)|}{\sqrt{\mu(x)+\nu(x)}}
\nonumber
\\
&\le&\sqrt{ \biggl(\sum_{x\in V}\mu(x)+\nu(x) \biggr)
\biggl(\sum_{x\in
V}\frac{(\mu(x)-\nu(x))^2}{\mu(x)+\nu(x)} \biggr)}\\
&=&\sqrt{2}
\Delta (\mu,\nu ).\nonumber
\end{eqnarray}
This quantity has an advantage compared to the total variation: for
any $f\dvtx G\rightarrow\mathbb R$, we have (using Cauchy--Schwarz similarly)
%
%e5 #&#
\begin{equation}
\label{important}\bigl |\mu(f)-\nu(f) \bigr|\le\Delta (\mu,\nu) \bigl(\mu
\bigl(f^2\bigr)+\nu\bigl(f^2\bigr) \bigr)^{1/2}.
\end{equation}
With the total variation, one would obtain a similar but weaker
inequality with the $L^\infty$-norm in place of the $L^2$-norm (the
former can in principle be much larger than the later). The reasons for
using $\Delta$ (rather than, say,
the total variation distance) will be discussed in more detail on page
\pageref{pg:Csiszar}, but most
readers would be better served by reading the paper linearly, that is,
first see how $\Delta$ is used to prove Theorem~\ref{no sublinear harmonic} and only then take a look at this discussion.

Let us introduce a convenient notation, used only in this section. Let
$\law(Z)$ denote the law of
a random variable $Z$, that is, the measure on the space of values of
$Z$ induced by it. If
$\mathcal E$ is some event, then we will denote by $\law(Z|{\mathcal
E})$ the law of $Z$
conditioned on $\mathcal E$ happening.

With this notation, we are now in a position to state an important
lemma, which is a quantitative version of the following well-known
fact: for any two random variables $X$ and $Y$,
$H(X,Y)\le H(X) + H(Y)$ with equality holding if and only if $X$ and
$Y$ are
independent.

%le7 #&#
\begin{lemma}\label{lem:XY}For any two random variables $X$ and $Y$,
%
%e6 #&#
\begin{equation}\quad
\label{eq:XY} \sum_yP(Y=y)\Delta^2
\bigl(\law(X),\law(X|Y=y)\bigr) \le2 \bigl(H(X)+H(Y)-H(X,Y) \bigr).
\end{equation}
\end{lemma}

%(recall that $\law(X)$ is the law of the variable $X$ i.e.\ the
%measure induced on its set of values)

\begin{pf}We first note that for $t > 0$,
%
%e7 #&#
\begin{equation}
\label{eq:convex} 2t\log t \ge\frac{(t-1)^2}{t+1}+2t-2
\end{equation}
%
%(one may verify the first inequality easily by noting that both sides
%are equal at 1 and the derivative of the left side is smaller than
%that of the right side).
[this can be seen by Taylor expanding $t \log t$ to the second order
at $1$, which gives that $t \log t = t-1 + \frac{(t-1)^2}{2 t^*}$ for
some $t^*$ in the interval between
$t$ and $1$, so $t^* \leq t+1$].
Denote
\[
p(x)=P(X=x),\qquad  p(y)=P(Y=y),\qquad p(x,y)=P(X=x,Y=y).
\]
Then, the left-hand side of \eqref{eq:XY} is [recall the definition
(\ref{eq:defDelta}) of $\Delta$]
\begin{eqnarray*}
\mbox{LHS}&=&\sum_yp(y)\sum
_x\frac{ (p(x,y)/p(y)-p(x)
)^2}{p(x,y)/p(y)+p(x)}\\
&=&\sum_{y,x}p(x)p(y)
\biggl(\frac{(
{p(x,y)}/{(p(x)p(y))}-1)^2}{{p(x,y)}/{(p(x)p(y))}+1}+
\smash {\underbrace {2\frac{p(x,y)}{p(x)p(y)}-2}_0}
\biggr),\phantom{\underbrace{
\frac{p}{p}}}
\end{eqnarray*}
where we were allowed to add the expression denoted by
\raisebox{8pt}[\height][8pt]{$\underbrace{}_0$} since summing over $x$
and $y$ makes these terms
cancel out (they are both equal to $2$). Using (\ref{eq:convex}) this gives
\begin{eqnarray*}
\mbox{LHS}&\stackrel{\scriptsize{(\ref{eq:convex})}} {\le}& 2\sum
_{x,y}p(x)p(y) \biggl(\frac{p(x,y)}{p(x)p(y)}\log\frac
{p(x,y)}{p(x)p(y)}
\biggr)
\\
&=& 2\sum_{x,y}p(x,y) \bigl(\log p(x,y)-\log p(x)-\log
p(y) \bigr)\\
&=&2\bigl(-H(X,Y)+H(X)+H(Y)\bigr),
\end{eqnarray*}
where in the last equality we used that $\sum_yp(x,y)=p(x)$ and $\sum_xp(x,y)=p(y)$.
\end{pf}

We will always be interested in the particular case of random walks. In
order to lighten the notation, we set
%
%e8 #&#
\begin{eqnarray}\label{eq:defDeltan}
\Delta_n(x,y)\dvtx &=&\Delta \bigl(\law(X_n|X_0=x),
\law(X_{n-1}|X_0=y) \bigr)
\nonumber
\\[-8pt]
\\[-8pt]
\nonumber
%\intertext{or equivalently, by Markov's property}
&=&\Delta \bigl(\law(X_n|X_0=x),
\law(X_{n}|X_1=y) \bigr),
\nonumber
\end{eqnarray}
the last equality following by the Markov property [recall that $\law
(X|{\mathcal E})$ denotes the law of $X$ conditioned
on $\mathcal E$].
Note that the second measure is the law of the random walk after $n-1$
steps, so the definition is not symmetric in $x$ and $y$.
%Here and below $\law(X)$ stands for the law of $X$ \emph{i.e.} the
%measure induced by $X$ on its value space.

Lemma~\ref{lem:XY} is used to proved the following theorem.

%th8 #&#
\begin{theorem}\label{entropy}
Let $(P,V,\rho)$ be a stationary environment. For every $n>0$, we have
%
%e9 #&#
\begin{equation}
\label{main inequality} \mathbb E \bigl(\Delta_n(\rho,X_1)^2
\bigr)\le2(H_{n}-H_{n-1})
\end{equation}
(as usual $\E$ is over both the environment and the
randomness of $X_1$).
\end{theorem}

Before proving Theorem~\ref{entropy}, we state a result from
\cite{BC10} concerning $H_{1,n}$. We isolate it from the rest of the
proof because it is the only place where stationarity is used
(stationarity replaces transitivity as used in the context of groups).

%le9 #&#
\begin{lemma}\label{increments entropy}
Let $(P,V,\rho)$ be a stationary environment. For every $n>0$, we have
$H_{1,n}=H_{n-1}+H_1$.
\end{lemma}

\begin{pf}
Fix $n>0$. %Invoking the definition of the mean entropy, we can write
%We first use the conditional entropy formula $H(X,Y)=H(X|Y)+H(Y)$. In
%our context this is simply
A simple computation leads to
\begin{eqnarray*}
&&\Hh_{1,n}(P,V,\rho)\\
&&\qquad=\sum_{x\sim\rho,y\in G}\phi \bigl(
\Pp_\rho(X_1=x,X_n=y) \bigr)
\\
&&\qquad=\sum_{x\sim\rho}\Pp_\rho(X_1=x)
\sum_{y\in G}\phi \bigl(\Pp _\rho(X_n=y
| X_1=x) \bigr)+\sum_{x\sim\rho}\phi \bigl(
\Pp_\rho(X_1=x) \bigr),
\end{eqnarray*}
which we simplify using the Markov property giving
\[
\Pp_\rho(X_n=y | X_1=x)=\Pp_x(X_{n-1}=y).
\]
Taking the expectation with respect to the environment we obtain
\begin{eqnarray*}
H_{1,n}&=&%\mathbb E\left[\sum_{y\sim\rho}\Pp_
%\rho(X_1=y)\sum_{z\in G}\phi\big(\Pp_\rho(X_n=z|X_1=y)\big) - \sum_{z
%\in G}\phi\big(\Pp_\rho(X_n=z)\big)\right]\\& =
\mathbb E \biggl[\sum
_{x\sim\rho}\Pp_\rho(X_1=x)\sum
_{y\in G}\phi \bigl(\Pp _x(X_{n-1}=y) \bigr)
\biggr]+\mathbb E \biggl[\sum_{x\sim\rho}\phi \bigl(\Pp
_\rho(X_1=x) \bigr) \biggr]
\\
&=&\mathbb E \bigl[\Hh_{n-1}(P,V,X_1) \bigr]+\mathbb E
\bigl[\Hh _1(P,V,\rho ) \bigr]=H_{n-1}+H_1,
\end{eqnarray*}
where in the last equality we used the fact that $(P,V,X_1)$ has the
same law as $(P,V,\rho)$ (this is not a property of entropy, it
would hold for any function of the environment).
%--- formally, we use the coupling that makes
%them isomorphic with probability 1, and since a Markov chain
%isomorphism preserves the entropy, the expected entropy of the two
%environments $(P,V,X_1)$ and $(P,V,\rho)$ must be equal (this is not a
%property of entropy, it would be true for any function of the
%environment).
\end{pf}

Before continuing, let us state one corollary of Lemma~\ref{increments
entropy} which is
not necessary for the proof of Theorem~\ref{entropy} but does shed
some light on the quantities involved.

%
%co10 #&#
\begin{corollary}\label{cor:decreasing} $H_n-H_{n-1}$ is decreasing.
\end{corollary}

In other words, the sequence $H_n$ is concave.

\begin{pf} By Lemma~\ref{increments entropy},
\[
H_n-H_{n-1}=H_n-H_{1,n}+H_1=
\E[\Hh_n-\Hh_{1,n}]+H_1.
\]
The quantity $\Hh_n-\Hh_{1,n}$ can be written as the conditioned entropy
$-H(X_1|X_n)$ where $X_n$ is the random walk at time $n$ (this
statement is quenched). This,
however, increases since
%
%e10 #&#
\begin{equation}
\label{eq:monot} \Hh(X_1|X_n)=\Hh(X_1|X_n,X_{n+1})
\le\Hh(X_1|X_{n+1}),
\end{equation}
where the equality is due to the fact that conditioned on $X_n$,
knowing $X_{n+1}$ gives you no information about what happened before
time $n$; that is, by the Markov property, conditional on $X_n$ we have
that $X_1$ is independent
of $X_{n+1}$.
The inequality in (\ref{eq:monot}) is a generic fact about
entropy---conditioning on more information reduces the relative
entropy [namely, $H(X|Y,Z) \leq H(X|Y)$ for any three random variables
$X$, $Y$ and $Z$]. % \eqref{eqn:H of X cond on Y,Z}.
Hence $\Hh_n-\Hh_{1,n}$ decreases, and so does its expectation.
\end{pf}

\begin{pf*}{Proof of Theorem~\ref{entropy}}
This is a direct corollary of Lemmas \ref{increments entropy} and
\ref{lem:XY}. Indeed, by Lemma~\ref{lem:XY},
\begin{eqnarray*}
\Ee \bigl(\Delta_n(\rho,X_1)^2 \bigr)&=&\sum
_x\Pp(X_1=x)\Delta\bigl(\law
(X_n),\law (X_n|X_1=x)\bigr)^2
\\
&\le&2(\Hh_1+\Hh_n-\Hh_{1,n}).
\end{eqnarray*}
We now take expectation with
respect to the environment and get from Lemma~\ref{increments entropy} that
\[
\E \bigl(\Delta_n(\rho,X_1)^2 \bigr) \le2\E(
\Hh_1+\Hh_n-\Hh _{1,n}) = 2(H_n-H_{n-1}).%\qedhere
\]
\upqed\end{pf*}

%\begin{corollary}
%Let $(G,\nu,\rho)$ be a stationary random graph with uniform
%polynomial growth. There exists $C>0$ such that
%\begin{equation}\mathbb E\big(\liminf_{n\rightarrow\infty} n\cdot
%\left(\Pp_\rho(\rho,X_1)\Delta_G^n(\rho, X_1)\right)^2\big) \le  C.
%\end{equation}
%\end{corollary}

We are now in a position to prove Theorem~\ref{no sublinear harmonic}.

\begin{pf*}{Proof of Theorem~\ref{no sublinear harmonic}}
We only need to prove that for almost every environment, $h(\rho
)=h(X_1)$ a.s., for
any sublinear harmonic function. Indeed, stationarity would then imply
that for
almost every $P$, $h(X_n)=h(X_{n+1})$ a.s. for any sublinear harmonic
function. Since the Markov chain is irreducible, $(X_n)$ can visit any
vertex, and we deduce that almost surely any sublinear harmonic
function is constant.

For any harmonic function $h$ with respect to the environment, we have
for all $x$ and $n$,
\[
h(x)=\Ee_x\bigl(h(X_n)\bigr). % h(X_1)=\Ee_{X_1}(h(X_{n-1})).
\]
We use this twice, once for $x=\rho$ and once for an arbitrary $x$ and
$n-1$. We get
\begin{eqnarray*}
\bigl|h(\rho)-h(x)\bigr |& =&\bigl|\Ee_\rho\bigl[h(X_n)\bigr]-
\Ee_{x}\bigl[h(X_{n-1})\bigr] \bigr|
\\
\mbox{by } (\ref{important})&\le&\Delta_n(\rho,x)\sqrt{
\Ee_\rho\bigl[h^2(X_n)\bigr]+\Ee_{x}
\bigl[h^2(X_{n-1})\bigr]}.
\end{eqnarray*}
We use this with $x=X_1$, integrate over $X_1$ and get
%Since $\Ee_x [ h^2(X_{n-1}) ] = \E_\rho[ h^2(X_n) \ | \ X_1=x ]$,
%taking expectation over $X_1$ gives
%
%e11 #&#
\begin{eqnarray}
\label{crucial inequality} \Ee_\rho \bigl|h(\rho)-h(X_1) \bigr| &\le&
\Ee_\rho \Bigl[\Delta_n(\rho,X_1)\sqrt{
\Ee_\rho\bigl[h^2(X_n)\bigr]+\Ee
_{X_1}\bigl[h^2(X_{n-1})\bigr]} \Bigr]
\nonumber
\\[-8pt]
\\[-8pt]
\nonumber
\mbox{by Cauchy--Schwarz} &\le& \sqrt{2\Ee_\rho \bigl[
\Delta_n(\rho,X_1)^2 \bigr]\Ee_\rho
\bigl[h^2(X_n)\bigr]},
\end{eqnarray}
where in the last line we also used that $\Ee_\rho[\Ee
_{X_1}[h^2(X_{n-1})]]=\Ee_\rho[h^2(X_n)]$.

By assumption, the Markov chain has annealed polynomial growth.
Therefore, the entropy satisfies
\[
H_n\leq\mathbb E\bigl[\log\bigl|\Bb_\rho(n)\bigr|\bigr]\leq\log
\mathbb E\bigl[\bigl|\Bb_\rho (n)\bigr|\bigr]\le\log\bigl[Cn^d\bigr]
\]
and is at most logarithmic (we used the fact that $\log$ is concave). Hence
$H_n-H_{n-1}\le c/n$ for infinitely many $n$. Using Theorem~\ref
{entropy} and (\ref{pg:SDB}) we get
\[
\E \bigl[n \Delta_n(\rho,X_1)^2 \bigr] + \E
\bigl[ n^{-1} \Dd(X_n,\rho )^2\bigr] \leq C\qquad
\mbox{for infinitely many $n$.}
\]
Hence, by Fatou's lemma, for almost every environment there exists
$\mathbf{c}
_1<\infty$ such that
%
%e12 #&#
\begin{equation}\qquad
\label{bound coupling} \Ee_\rho\bigl[n \Delta_n(
\rho,X_1)^2\bigr] + \Ee_\rho\bigl[
n^{-1} \Dd(X_n,\rho)^2\bigr]\leq
\mathbf{c}_1\qquad \mbox{for infinitely many $n$,}
\end{equation}
where this time the sequence of $n$ for
which it holds depends on the environment, that is, is random.

Now, assume that $h$ has sublinear growth. For any $\varepsilon>0$,
there exists a constant ${\mathbf K}$ such that for all $x\in V$,
%
%e13 #&#
\begin{equation}
\label{bound h}h^2(x)\leq\varepsilon\Dd(x,\rho )^2+{
\mathbf K}.
\end{equation}

%Third, the random walk satisfies \SDB{}, implying that there exists $
%\cc_2(P)<\infty$ (not depending on $\e$) such that for every $n>0$,
%\begin{equation}\label{bound speed}\Ee_\rho[\Dd(X_n,\rho)^2]\leq
%\cc_2n.\end{equation}
Putting \eqref{bound h} and \eqref{bound coupling}
%and \eqref{bound speed}
in \eqref{crucial inequality}, we deduce that for almost every
environment, and for every $h$ harmonic and sublinear on it,
% for an infinite number of $n$,
%\begin{equation}\label{22}\Ee_\rho[h^2(X_n)] \le \varepsilon\Ee_\rho[
%\Dd(X_n,\rho)^2]+K \le \varepsilon C_2n+K.\end{equation}
%Putting \eqref{21} and \eqref{22} together, we obtain when $n$ goes to
%infinity that
%
\[
\Ee_\rho\bigl(\bigl|h(\rho)-h(X_1)\bigr|\bigr)\le
\mathbf{c}_2 \varepsilon^{1/2}.
\]
Letting $\varepsilon$ go to 0, we deduce that $h(\rho)=h(X_1)$ almost
surely for any sublinear harmonic function.
\end{pf*}

Inequality \eqref{crucial inequality} relates the entropy to the value
of possible harmonic functions at~$X_n$. Its use is not restricted to
the case of diffusive environments with polynomial growth. For
instance, one can use this inequality to prove a characterization of
almost sure Liouville property for stationary random graphs (this was
proved in~\cite{BC10} using a more direct generalization of~\cite{KV}).
For completeness, we state the result in \cite{BC10} here.

%co11 #&#
\begin{corollary}[(\cite{BC10})]
Let $(P,V,\rho)$ be a stationary environment. If $H_n/n$ converges to
0, then $P$ has the Liouville property (i.e., has no nonconstant
bounded harmonic functions) almost surely.
\end{corollary}

%\refstepcounter{phoney}
\label{pg:Csiszar}
We would like to emphasize why we use $\Delta(\mu,\nu)$. Csisz\'ar's
inequality \mbox{\cite{Csi63,Csi66}} relates the
total variation between two measures to their relative entropy. In our
context, an inequality involving the total variation can also be
found, hence giving a bound on the best coupling (in time) between two
random walks starting at neighbors. For completeness, we state the
inequality here [it is a consequence of (\ref{eq:TVDelta}) applied to
\eqref{main inequality}]: for a stationary environment $(P,V,\rho)$
and $n>0$, we have
\[
\mathbb E \bigl(\bigl\Vert \law(X_n)-\law(X_n|X_1)
\bigr\Vert _{\mathrm{TV}}^2 \bigr)\le4 (H_{n}-H_{n-1}).
\]
Interestingly, this inequality is not strong enough for our
applications, since controlling the probability that two random walks
merge before time $n$ says nothing about their behavior when they do
not couple.

%pa2.subsection.subsubsection.1 #&#
\subsection*{Other growth rates}
The same
argument as in Theorem~\ref{no sublinear harmonic} can also be used
with growth rates bigger than
polynomial. A general statement would be the following.

%
%th1 #&#
{\renewcommand{\thetheoremmm}{3$'$}
\begin{theoremmm}\label{page:2p}
%Let $(P,V,\rho)$ be a stationary environment. Then
%there are no nonconstant harmonic functions $h$ which satisfy
%\[
%\Ee_\rho[h(X_n)^2]\cdot(H_n-H_{n-1})\to0
%\]
%even if this holds just along a subsequence of $n$.
Let $\P$ be the measure of a stationary environment $(P,V,\rho)$.
For every (nonrandom) sequence $(n_k)_k$ with $n_k \to\infty$,
we have that $\P$-a.s. there does not exist a nonconstant harmonic function
$h\dvtx V \to\R$ such that
\[
\Ee_{\rho} \bigl[ h(X_{n_k})^2 \bigr]
\cdot(H_{n_k} - H_{{n_k}-1}) \to0. %
\]
\end{theoremmm}}

In particular this holds for fixed transitive graphs, which is a
version of a result of \cite{EK10}, Section~5.

%pa2.subsection.subsubsection.2 #&#
\subsection*{Examples} We finish this section by presenting a collection
of examples.

%ex2.1 #&#
\begin{example}[(Random conductance)]\label{random conductance}
Consider the graph $\Z^d$, and let $\nu$ be given by a shift-invariant
law (e.g., i.i.d. positive random variables). We assume that
the set of sites connected by edges with positive conductances is
infinite. The random walk induces a Markov process on the environment
(cf. Kipnis and Varadhan \cite{KV86}), called the environment as seen
from the particle. This process can be made stationary by weighting
each configuration proportionally to $\nu(\rho)$.% (the root $\rho$ in
%our formalism) according to the stationary measure for the random walk.

This model has been studied extensively. Under the assumption of
\emph{uniform ellipticity}: $\exists\alpha>0\dvtx \mathbb
P[\alpha<\nu(x,y)<1/\alpha]=1$, many things are known on the
environment. First, the Poincar\'e inequality is a direct consequence
of the $\Z^d$ case. Second, Delmotte proved in \cite{Del99} that the
Poincar\'e inequality implies that there exist $c_1,c_2>0$ such that
\[
\Pp_\rho[X_n=x]<\frac{c_1}{n^{d/2}}e^{-c_2\Dd(x,\rho)^2/t}
\]
(a corresponding lower bound also holds but is not needed for our
purposes). Third, an annealed invariance principle holds in the sense
that the law of the paths under the measure integrated over the
environment scales to a nondegenerate Brownian motion \cite{KV86}. In
particular, Theorem~\ref{no sublinear harmonic} applies in this case.

Once the assumption of uniform ellipticity is relaxed, matters get
more complicated. An example of random conductance models without
uniform ellipticity is the infinite cluster of percolation which we
will discuss next. For an unusual example of a
transitive conductance model, see the work of Disertori, Spencer and
Zirnbauer \cite{DSZ10} who reduced a supersymmetric hyperbolic sigma
model to the study of random walk on a certain (highly correlated) random
environment.
\end{example}

%ex2.2 #&#
\begin{example}[(Infinite cluster of percolation)]\label{exmp:percolation}
Consider the percolation measure with a parameter $p$ such that there
exists an infinite cluster with probability~1. See
\cite{G99} for details about percolation. Set $\mathbb P_0$ to be the
law of the infinite cluster conditioned to contain 0. As in the
previous example, the random walk on $\omega$ induces a Markov chain on
the space
$\Omega$ of infinite subgraphs of $\Z^d$ containing the origin. When
weighting each configuration proportionally to the number of neighbors
of the origin we obtain a stationary measure with respect to the shift
along the random walk.

Since the infinite cluster of percolation can be seen as a stationary
random graph with polynomial volume growth and since the random walk is
diffusive \cite{Ke86,Bar04}, Theorem~\ref{no sublinear harmonic}
applies, and we get Theorem~\ref{thm:percolation}.
%\begin{corollary}
%Almost surely, there does not exist any sublinear harmonic function on
%the infinite cluster of supercritical percolation on $\Z^d$.
%\end{corollary}
\end{example}

%ex2.3 #&#
\begin{example}[(Centered random environments)]This is our first
nonreversible example. A centered random environment is, roughly
speaking, a Markov chain on $\Z^d$ such that the probabilities can
be ``decomposed'' into a sum over cycles. Such environments, even
when nonreversible, are still heuristically quite close to
reversible, and in particular they have a stationary version which
is related to the usual version by an explicit reweighting, like in
the reversible case~\cite{DK08}, Section~3. See Deuschel and K\"osters
\cite{DK08} for a proof of a CLT, which
implies~(\ref{pg:SDB})---of course, a CLT is much stronger than (\ref{pg:SDB}). Hence,
our results can be applied in this context as well.
\end{example}

%ex2.4 #&#
\begin{example}[(Balanced random environments)]This is another
nonreversible example, which is ``farther'' from reversible than
the previous one. A balanced random environment is a Markov
chain $P$ with state space $\Z^d$ and nearest neighbor movements, such
that for every $x\in\Z^d$ and every unit vector $e_i$,
$P(x,x+e_i)=P(x,x-e_i)$. It follows that $X_n$ is a martingale, and
hence (\ref{pg:SDB}) is an immediate corollary of the Azuma--Hoeffding
inequality. The issue is therefore only stationarity. In the case
that the environment $\mu$ is uniformly elliptic and stationary and
ergodic \emph{to the action of
$\Z^d$} (this is different from our notion of stationarity!), Lawler
showed that there exists a stationary measure
(in our
sense) $\lambda$ which is mutually absolutely continuous with
respect to $\mu$; see \cite{L82}, Theorem~3. Hence our results apply to $\lambda$, and hence
also to $\mu$. Guo and Zeitouni weakened the requirement of uniform
ellipticity to just ellipticity, at the price of restricting the
environment to the i.i.d. case \cite{GZ10}. Berger and Deuschel
\cite{BD11} have removed the requirement of ellipticity
altogether in the i.i.d. case.
\end{example}

%ex2.5 #&#
\begin{example}[(Random environments with cut points)] Under certain
conditions, one can prove that a random walk in nonreversible
random environments in~$\Z^d$, $d$ large enough, has cut points, and
deduce from that a CLT and
the existence of a stationary environment, hence our techniques
apply. See \cite{BSZ03} for the details.
\end{example}

Let us give one example which is not embedded in $\Z^d$, and in fact
has unbounded degrees.

%ex2.6 #&#
\begin{example}[(Poisson point process)]\label{exmp:poisson} Examine a
Poisson point process
in~$\R^d$. Add the point 0 (this is often called ``the Palm
process''), and let it be the root. Construct a graph by some process
invariant under translations of $\R^d$. For example, connect any two
points by an edge with weight which depends on their Euclidean
distance \cite{CFP} or construct the Delauney triangulation
\cite{FGG}. Give each configuration a ``probability proportional to
the total weight of 0.'' The resulting process is stationary and
diffusive; see, for example,  \cite{CFP}, Section~2.1 or \cite{FGG}, Lemma
A.1, for
stationarity---subdiffusivity can be deduced from \cite{BP89},
Section~2, or from the
two previous papers. Hence our theorem applies.
\end{example}

The previous examples dealt with random walks which are
diffusive. An interesting situation, which cannot hold in the case of groups,
is environments with subdiffusive behavior. We give four examples of these.

%ex2.7 #&#
\begin{example}[(Graphical fractals)] A graphical fractal is a graph
which is constructed like one of the classical fractals (the
Sierpinski gasket, e.g.), but inside out---bigger pieces of
the graph are constructed from smaller pieces by connecting them in
a repeated fashion; see \cite{Bar} for precise definitions and main
properties. See Figure~\ref{cap:gasket} for an example, the
graphical Sierpinski gasket. A graphical fractal always has an
invariant measure and is
always diffusive or subdiffusive, and in many examples is
in fact subdiffusive; see, for example, \cite{B04}. Let us
remark that a significant part in the remarkable work of Barlow and
Bass on the Sierpinski carpet \cite{BB99} has to do with the
construction of a coupling. Therefore, a tool (like the one described
in this section) that gives easy proofs that
couplings exist should be useful.
\end{example}

%ex2.8 #&#
\begin{example}[(Critical Galton--Watson trees)]\label{exmp:GW}
The critical Galton--\break Watson tree with any offspring distribution
conditioned to survive is stationary (see \cite{KS10,LP,LPP95}) and
subdiffusive.
If the offspring distribution has finite variance,
the diffusivity exponent $\nicefrac{1}{3}$ was proved in
\cite{Ke86}. Thus Theorem \ref{page:2p} applies in
this case, and we get that it has no harmonic function of growth $o(\Dd
(\rho,x)^{3/2})$.

%f1 #&#
\begin{figure}

\includegraphics{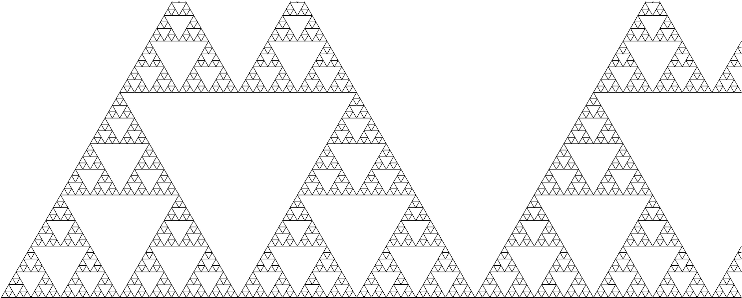}

\caption{A portion of the graphical Sierpinski
gasket.} \label{cap:gasket}
\end{figure}

This example is not so impressive since (as it is well known) this graph
has infinitely many cut-edges between the root and infinity, and therefore
the only harmonic functions (without any growth restrictions) are the
constants. However, the cut-edges argument fails after even slight
variations, while Theorem~\ref{no sublinear harmonic} is
robust. Examples include taking a product of a Galton--Watson tree with
a finite graph or with itself. The same remark applies to the next example.
% However, if we take the graph product of the critical Galton-Watson
%tree (conditioned to survive)
% with itself, we get a sub-diffusive stationary random graph with
%nontrivial harmonic functions.
% Our results imply that in this case there are no linear growth
%harmonic functions a.s.
\end{example}

%ex2.9 #&#
\begin{example}[(Infinite incipient cluster)]
Consider critical percolation on $\Z^d$ conditioned on the fact
that the origin is connected to infinity \cite{Ke86a}. Conditioning on
this event, which has probability 0 (proved
in $d=2$ and high $d$ and conjectured in the others), requires some
care. Nevertheless, the object
can be defined properly using a limit\vspace*{1pt} process. For example, one may
take $p_c+\epsilon$ percolation, condition on~$\vec{0}$ being in the
cluster and then take a limit of
the resulting measures as $\epsilon\to0$. Since for each $\epsilon$
the measure is
stationary (as usual after reweighting the configurations
proportionally to the degree of $\vec{0}$), so will be their limit
if it exists (or any subsequence limit in general). The limit is known
to exist in two dimensions \cite{Ke86a,AJ} and in high dimensions
\cite{HJ04,HHH}.
It was
proved in \cite{Ke86,KN09} that the random walk is subdiffusive on this
cluster (in high dimension the diffusivity exponent is~$\nicefrac{1}{3}$, as on the tree). Since it is embedded in $\Z^d$,
it grows no faster than polynomially and the results may be applied in
this context.
%As in the Galton-Watson case, this graph has infinitely many cut-edges
%separating the
%origin from infinity, so the only harmonic functions on it are the
%constant functions.
%The IIC can be seen as a stationary random
%graph \cite{BC10}. \gady{it's not even mentioned there}
%If we take the graph product of this graph with itself, then the
%diffusivity exponent does not
%change, and polynomial growth is preserved.
%Therefore, there are no linear growth harmonic functions on the
%product as well.
\end{example}

%ex2.10 #&#
\begin{example}[(Graph limits and UIPQ)]\label{exmp:UIPQ}
Let $G_n$ be fixed or random finite graphs. Take $\rho_n$ to be a random
vertex in $G_n$, selected according to the stationary measure on
$G_n$. Then the limit of $(G_n,\rho_n)$, if it exists, is called the
graph limit~\cite{BS01}. This limit is always stationary, \cite{Kri},
Section~1.3.

A particular case is provided by a uniformly chosen planar
quadrangulation $G_n$ with $n$
faces. The graph limit is known as the
uniform infinite planar quadrangulation.
It is well known to be of polynomial growth \cite{CD06}.
In \cite{BC11}, it was proved to be
subdiffusive with diffusivity exponent bounded from above by $\nicefrac
{1}{3}$. Thus there are no linear growth harmonic functions in
this case either.
\end{example}

%pa2.subsection.subsubsection.3 #&#
\textit{A remark on connectivity.}
We assumed throughout that the environment $(P,V,\rho)$ is irreducible,
that is, that for any $v,w\in V$ there is some $n$ such that
$P^n(v,w)>0$. This assumption was only used once: we showed that a
not-necessarily-irreducible stationary environment satisfies that
every harmonic function $h$ has $h(\rho)=h(X_1)$ almost surely, and
concluded, using irreducibility, that $h$ is constant. The assumption
of irreducibility is of course necessary, as a disconnected graph
always has bounded nonconstant harmonic functions, namely functions
which are constant on each component, but with different values.

Nevertheless, in the nonreversible case, the assumption of
irreducibility can be weakened slightly: we only need to assume that
for every $v$ and $w$ there exist $n,m$ and $x$ such that
$P^n(v,x)>0$ and $P^m(w,x)>0$. The proof is the same---since
$h(\rho)=h(X_1)$ almost surely then this gives that $h(v)=h(x)=h(w)$
almost surely and $h$ is constant. The following stationary graph
provides a simple example. Take a 3-regular tree $T$.
%choose one
%bi-infinite simple path $\gamma$ on this tree, and oriented all edges
%outside $\gamma$ in the direction of $\gamma$, and all edges in
%$\gamma$ toward one of the ends of $\gamma$. It is easy to see that
%the resulting environment is stationary
Choose a height function $\ell$ (i.e., a function such that each vertex
has one neighbor with $\ell$ bigger by one, and two neighbors with
$\ell$ smaller by one), and orient all edges ``up,'' that is, in the
direction of the larger $\ell$. Of course, the random walk on the
resulting graph is so degenerate it can hardly
be called random, as each vertex has only one outgoing edge. But this
is irrelevant at this point.
This
environment is not irreducible in the usual sense, but does satisfy
the weaker assumption and hence our results apply (again, in this case
it is simple to analyze the harmonic functions directly).
Taking the graph product with $\mathbb{Z}$ will yield a slightly less
trivial example.

%s3 #&#
\section{Polynomial growth harmonic functions}\label{sec:polynomial growth}

In this section we prove Theorem~\ref{finite dimension}. The proof
boils down to the observation that macroscopic Poincar\'e inequality
and volume growth estimates are sufficient. The strategy follows the
lines of Shalom and Tao \cite{ST,T10}, where a quantitative version of
Gromov's theorem on groups of polynomial growth (any group
of polynomial growth is virtually nilpotent) is proved. The proof is
inspired by an elegant proof of this theorem due to Kleiner
\cite{Kle10} utilizing spaces of harmonic functions with polynomial
growth in a crucial way. We start with a very general inequality,
called the reverse Poincar\'e inequality, which holds in any
graph. For the sake of completeness, we prove it in our context.

%pr12 #&#
\begin{proposition}[(Reverse Poincar\'e inequality)]\label{reverse Poincare}
For any weighted graph $(G,\nu)$
and any function $h\dvtx G\rightarrow\mathbb R$ harmonic on a ball $\Bb_x(2n)$,
%
%e14 #&#
\begin{equation}
\sum_{(y,z) \in E(\Bb_x(n))} \bigl(h(z)-h(y) \bigr)^2
\nu(y,z)\le\frac
{4}{n^2}\sum_{y \in\Bb_x(2n)}h(y)^2
\nu(y)
\end{equation}
for every $x\in G$ and $n>0$.
\end{proposition}

\begin{pf}
For this proof, we denote the quantity $f(x)$ by $f_x$. Let
$h\dvtx G\rightarrow\mathbb R$ be harmonic on $\Bb_x(2n)$, and
let $\phi$ be a function such that $\phi_y = 1$ for $y \in\Bb_x(n)$,
$\phi_y = 0$ for $y \notin\Bb_x(2n-1)$ and $|\phi_y - \phi_z| \leq
1/n$ for all $y \sim z$.
For example,
\[
\phi_y:=\min \biggl(1,2-\frac{\Dd(y, x)}{n} \biggr)\qquad \mbox{for any }y
\in\Bb_x(2n). %
\]
%
%$$ \phi_y: = \left\{ \begin{array}{lr}
%
%\end{array} \right. $$
We have
%
%e15 #&#
\begin{equation}
\label{eq:justaddphi} \sum_{E(\Bb_x(n))} (h_y-h_z
)^2\nu(y,z)=\sum_{E(\Bb
_x(n))}\frac{1}2
\bigl(\phi_y^2+\phi_z^2\bigr)
(h_y-h_z )^2\nu(y,z).
\end{equation}
To make the calculation a little shorter we represent the sum on the
right-hand side of (\ref{eq:justaddphi}) as a sum of $\frac{1}2
\phi_y^2(h_y-h_z)^2$ over \emph{directed edges}. Denote by $E^*$ the
set of directed edges in $B_x(2n)$, that is, both $(y,z)$ and $(z,y)$
appear in $E^*$ and are different. For an edge $(y,z)\in E^*$, a
straightforward (if a little lengthy) computation shows that
$\phi^2_y(h_z-h_y)^2$ is equal to the quantity
\[
\bigl(h_z\phi^2_z-h_y
\phi^2_y\bigr) (h_z-h_y)-h_z(
\phi_z-\phi _y)^2(h_z-h_y)-2h_z
\phi_y(\phi_z-\phi_y) (h_z-h_y).
\]

We start by dealing with the first term. Rearranging the sum [using
the fact that $h\phi^2$ vanishes outside $\Bb_x(2n-1)$ to add the
missing terms on the boundary] gives
\[
\sum_{E^*}\bigl(h_z
\phi^2_z-h_y\phi^2_y
\bigr) (h_z-h_y)\nu(y,z)= 2\sum
_{y\in\Bb_x(2n-1)}h_y\phi^2_y
\biggl(\sum_{z\sim y}( h_y-h_z)
\nu(z,y) \biggr).
\]
Since $h$ is harmonic, this sum equals 0.

For the second term, since $|h_z(h_z-h_y)|\le\frac{3}2 h_z^2+\frac{1}2
h_y^2$ and $|\phi_z-\phi_y|\leq1/n$, we have that
each summand is bounded by $(3h^2_z+h^2_y)/(2n^2)$.
When summing over $E^*$ we obtain
\[
\biggl\llvert \sum_{E^*}h_z(
\phi_z-\phi_y)^2(h_z-h_y)
\nu(y,z)\biggr\rrvert \le \frac{2}{n^2}\sum_{y\in\Bb_x(2n)}h^2_y
\nu(y).
\]

For the third term, note that
%
%e16 #&#
\begin{equation}
\label{eq:1/41}\bigl |h_z\phi_y(\phi_z-
\phi_y) (h_z-h_y) \bigr|\le\tfrac
{1}{4}(h_y-h_z)^2
\phi ^2_y+ h^2_z(
\phi_z-\phi_y)^2.
\end{equation}
So,
\begin{eqnarray*}
&&\sum_{E^*} \bigl|h_z
\phi_y(\phi_z-\phi_y) (h_z-h_y)
\bigr|\nu(z,y)\qquad
\\
&&\qquad\mbox{by (\ref{eq:1/41})}\le \frac{1}4\sum
_{E^*}(h_y-h_z)^2
\phi_y^2\nu(y,z)+ \sum_{E^*}h^2_z(
\phi_z-\phi_y)^2\nu(y,z)
\\
&&\qquad\hphantom{\mbox{By (\ref{eq:1/41})}}\le \frac{1}4\sum_{E^*}(h_y-h_z)^2
\phi_y^2\nu(y,z)+\frac
{1}{n^2}\sum
_{\Bb
_x(2n)}h^2_y\nu(y)
\end{eqnarray*}
using the bound $|\phi_z-\phi_y|\le\frac{1}n$ for every $y \sim z$.
Putting the bound on the different terms together leads to
\[
\sum_{E^*}(h_y-h_z)^2
\phi_y^2\nu(y,z) \le\frac{1}2\sum
_{E^*}(h_y-h_z)^2
\phi_y^2\nu(y,z)+\frac{4}{n^2}\sum
_{\Bb
_x(2n)}h^2_y\nu(y),
\]
which gives
\[
\sum_{E(\Bb_x(n))}(h_z-h_y)^2
\nu(y,z) \le\frac{1}2\sum_{E^*}(h_z-h_y)^2
\phi_y^2\nu(y,z)\le\frac
{4}{n^2}\sum
_{\Bb
_x(2n)}h^2_y\nu(y). %\qedhere
\]
\upqed\end{pf}

%le13 #&#
\begin{lemma}
Let $(G,\nu,\rho)$ be a rooted graph satisfying the
volume doubling condition $(\mathit{VD})_G$. Then there exists
$\mathbf{c}>0$ such that the following holds. For any $\lambda<
\infty$,
there exist $M_\lambda$ and $n_0$ such that for all $n>n_0$,
there is a covering of the ball $\Bb_\rho(\lambda n)$ by
$k < M_\lambda$ balls $\Bb_{y_1}(n),\dotsc,\Bb_{y_k}(n)$ satisfying that
every point $x\in\Bb_\rho(n)$ belongs to at most $\mathbf{c}$ balls
$\Bb_{y_i}(2n)$.

Furthermore, $\mathbf{c}$ depends only on the volume doubling constant
$\CVD$,
and $M_\lambda$ depends only on $\lambda$ and $\CVD$.
\end{lemma}

We call a covering with this property \emph{proper}. %Think of $\cc_2$
%as being the constant in $(P)_G$.

\begin{pf}
%Let $\e>0$, $\cc_2>0$ and $G$, consider $n$ large enough so that
%\VdG{} holds true with $\lambda=4/\e$\dvtx for every $x\in
% \Bb_\rho(n)$,
%$$\cc_1 n^d \le \nu(\Bb_x(\e n/4)) \le  \Cc_1 n^d.$$
%
%Find a collection of $k$ disjoint balls $\Bb_{y_1}(\e n/4)$,\dotsc,$
%\Bb_{y_k}(\e n/4)$ so that each point of $\Bb_\rho(n)$ is at distance $
%\e n/2$ of one of the balls. This collection has the following
%properties:
%\begin{itemize}
%\item$\Bb_{y_1}(\e n)$,\dotsc, $\Bb_{y_k}(\e n)$ is a covering of $
%\Bb_\rho(n)$,
%
%\item$k\leq\Cc_1/\cc_1(4/\e)^d$: indeed using $(V_d)_G$, $\mu(
%\Bb_{y_i}(\e n/4))\ge\cc_1 (\e n/4)^d$ for every $i$ and $\mu(\Bb_
%\rho(n))\le\Cc_1 n^d$. Since the balls are disjoint, it implies the
%inequality.
%
%\item At most $(4\cc_2)^d\Cc_1/\cc_1$ balls overlap: indeed, for every
%$x\in\Bb_\rho(n)$, $\mu(\Bb_x(\cc_2\e n))\le\Cc_1(\cc_2\e n)^d$, so
%that at most $(4\cc_2)^d\Cc_1/\cc_1$ points $y_i$ can belong to $\Bb_x(
%\cc_2\e n)^d$.
%\end{itemize}
%This proves the claim with $\cc:=\Cc_1/\cc_1(4\cc_2)^d$ and $M_
%\varepsilon:=\Cc_1/\cc_1(\cc_2/\varepsilon)^d$.
Let $\lambda$ and $G$ be as above. %Let $\cc_3 = (\cc_2 + \frac{1}{4})
%\e$,
Let $n$ be large enough so that $(\mathit{VD})_G$ holds for $2\lambda$ and $n/2$.
%% that give
%for all $x \in\Bb_\rho(\labda n)$,
%$$ \nu(\Bb_x(\e n/4) ) \geq\cc_1 (\e n /4)^d,
%\nu(\Bb_x(\cc_3 n)) \leq\Cc_1 (\cc_3 n)^d, $$
%and $\nu(\Bb_\rho(n)) \leq\Cc_1 n^d$.
Given this, we can choose a maximal family of disjoint balls
$\Bb_{y_1}(n/2), \ldots, \Bb_{y_k}( n/2)$
with $y_j \in\Bb_\rho(\lambda n)$ for all $j$:
%such that every vertex in $\Bb_\rho(n)$ is at distance at most $\e
%n/2$ of one of these $k$ balls.

\begin{itemize}
\item Since the family $\{\Bb_{y_j}(n/2)\}$ is maximal, every vertex
in $\Bb_\rho(\lambda n)$ must be within distance $\le n$ from one of
the $y_j$, so $\Bb_\rho(\lambda n)$ is covered by $\Bb_{y_1}( n),
\ldots, \Bb_{y_k}(n)$.

\item For any $x \in\Bb_\rho(\lambda n)$, if $x \in\Bb_{y_j}(2n)$,
then $\Bb_{y_j}( n/2) \subset\Bb_x(3 n)$.
Using volume doubling we see that $\nu(\Bb_x(3 n)) \leq\CVD^4 \nu
(\Bb
_{y_j}(n/2))$,
hence (since these balls are disjoint) %and since $\nu(\Bb_{y_j}(\e
%n/4)) \geq\cc_1 (\e n /4)^d$,
we have that the number of $y_j$ such that $x\in\Bb_{y_j}(2n)$ is at
most $\CVD^4$.
%$\cc: = \frac{\Cc_1}{\cc_1} \cdot(\frac{4 \cc_3}{\e} )^d = \frac{
%\Cc_1}{\cc_1} \cdot(4 \cc_2 +1)$.

\item Using the volume doubling similarly, we get that
$
\nu(\Bb_\rho((\lambda+1)n)) \le\break  \Cc\nu(\Bb_{y_j}(n/2))
$ for any $j$
(the constant is $\CVD^{\lceil\log_2(\lambda+1)\rceil+2}$).
Since these balls are all disjoint and fully contained in
$\Bb_\rho((\lambda+1)n)$, we get
\begin{eqnarray*}
k \min_j \nu\bigl(\Bb_{y_j}(n/2)\bigr) &\le& \nu
\biggl(\bigcup_j\Bb_{y_j}(n/2) \biggr)
\le\nu\bigl(\Bb_\rho\bigl((\lambda +1)n\bigr)\bigr)\\
&\le& \Cc\min
_j\nu\bigl(\Bb_{y_j}(n/2)\bigr),
\end{eqnarray*}
%
%, since $\nu(\Bb_{y_j}(\e n/4)) \geq\cc_1 (\e n /4)^d$,
%and since $\nu(\Bb_\rho(n)) \leq\Cc_1 n^d$, we get that the number of
%balls is at most
%$k \leq M_\e: = \frac{\Cc_1}{\cc_1} \cdot(\frac{4}{\e})^d$.
and we get that the number of balls $k$ is bounded by the same $\Cc$.\quad\qed%
%\qedhere
\end{itemize}
\noqed\end{pf}

%le14 #&#
\begin{lemma}\label{lemma b}
Let $(G,\nu,\rho)$ be a rooted graph satisfying $(P)_G$. Then there
exists a $\mathbf{c}>0$ such that
for every $\varepsilon>0$ and $n$ large enough, and for every proper
covering of $\Bb_\rho(n)$ by balls of radius $\varepsilon n$, if
$h\dvtx G\rightarrow\mathbb R$ is harmonic and has 0 mean on all the balls
of the covering, then
%
%e17 #&#
\begin{equation}
\sum_{z\in\Bb_\rho(n)}h(z)^2\nu(z)\leq\mathbf{c}
\varepsilon ^2 \sum_{z\in\Bb
_\rho(4 n)}h(z)^2
\nu(z).
\end{equation}

Further, $\mathbf{c}$ depends only on $\CP$, the constant in the
Poincar\'e
inequality and on the constants in the definition of a proper cover.
\end{lemma}

\begin{pf}
Fix $n$ large enough so that
$(P)_G$ holds true for $\lambda=1/\varepsilon$ and $\varepsilon n$,
that is, such that for
every $x\in
\Bb_\rho(n)$ and $f$ a map on $\Bb_\rho(n)$,
\[
\sum_{y\in\Bb_x(\varepsilon n)} \bigl(f(y)-\overline{f}_{\Bb
_x(\varepsilon n)}
\bigr)^2\nu (y)\leq \CP(\varepsilon n)^2\sum
_{(y,z)\in E(\Bb_x(2\varepsilon
n))}\bigl|f(y)-f(z)\bigr|^2\nu(y,z).
\]

Let $h\dvtx G\rightarrow\mathbb R$ be the harmonic function and
$\Bb_{y_1}(\varepsilon n),\dotsc,\Bb_{y_k}(\varepsilon n)$ be the
proper covering of
$\Bb_\rho(n)$ from the statement of the lemma. The hypothesis asserts that
$\overline{h}_{\Bb_{y_i}(\varepsilon n)}=0$ for every $i$, so that
Poincar\'e
inequality implies
\begin{eqnarray*}
\sum_{\Bb_{y_i}(\varepsilon n)}h^2(z)\nu(z)&=&\sum
_{\Bb
_{y_i}(\varepsilon
n)}\bigl(h(z)-\overline{h}_{\Bb_{y_i}(\varepsilon n)}
\bigr)^2\nu(z)
\\
&\le& \CP \varepsilon ^2n^2\sum
_{E(\Bb_{y_i}(2\varepsilon n))}\bigl(h(z)-h(t)\bigr)^2\nu(z,t).
\end{eqnarray*}
Since the $\Bb_{y_i}(2\varepsilon n)$ have uniformly bounded overlap
(each point
belong to at most~$\mathbf{c}$ balls), and since $\Bb
_{y_i}(2\varepsilon n)\subset
\Bb_\rho(2n)$, we find
%
%e18 #&#
\begin{eqnarray}
&&\sum_{\Bb_\rho(n)}h^2(z)\nu(z)
\nonumber
\\[-8pt]
\\[-8pt]
\nonumber
&&\qquad\le\mathbf{c}\CP
\varepsilon^2 n^2\sum_{E(\Bb
_\rho(2 n))}
\bigl(h(z)-h(t)\bigr)^2\nu(z,t).
\end{eqnarray}
Using the reverse Poincar\'e inequality (Proposition~\ref{reverse
Poincare}) for the larger ball, we conclude
%
%e19 #&#
\begin{equation}
\sum_{\Bb_\rho(n)}h^2(z)\nu(z)\le4\mathbf{c}\CP
\varepsilon^2 \sum_{\Bb_\rho(4
n)}h^2(z)
\nu(z),
\end{equation}
which implies the claim with the constant in the statement of the
lemma being~$4\mathbf{c}\CP$.
\end{pf}

\begin{pf*}{Proof of Theorem~\ref{finite dimension}}
We aim to prove
that the space of harmonic functions $u$ such that $|u(x)|\leq
\Cc\Dd(\rho,x)^k$ for every $x\in G$ is finite dimensional. Consider
a rooted graph $G$ satisfying $(\mathit{VD})_G$ and $(P)_G$. Let
$\mathbf{c}$ be large enough so that the two previous lemmas hold true.
On the set of harmonic functions on $\Bb_\rho(n)$, a scalar product
between two functions can be defined by
\[
\langle f,g \rangle_n=\sum_{\Bb_\rho(n)}f(x)g(x)
\nu(x).
\]
Consider $d$ harmonic functions $u_1,\dotsc, u_d$ on $G$ and set
$V=\mathrm{span}(u_1,\dotsc,u_d)$. Our goal is to compare
$\langle\cdot,\cdot\rangle_n$ and $\langle\cdot,\cdot\rangle
_{4n}$ for
these functions.

Let $\varepsilon>0$ be some parameter to be fixed later. For $n$ large enough,
there exists a proper covering $\Bb_{y_1}(\varepsilon
n),\dotsc,\Bb_{y_M}(\varepsilon n)$ of $\Bb_\rho(n)$ by
$M=M_{1/\varepsilon}$
balls. Therefore there is a codimension $d-M$ vector space $V_0\subset
V$ of
harmonic functions with mean 0 on each of the balls $\Bb
_{y_i}(\varepsilon
n)$. Let $v_1,\dotsc,v_d$ be an orthogonal basis of $V$ for
$ \langle\cdot,\cdot \rangle_{4n}$ such that
$v_1,\dotsc,v_{d-M}$ is a basis of $V_0$. Examine the Gram
matrix of $\{v_i\}$, that is, the $d\times
d$ matrix whose entries are $\langle v_i,v_j\rangle_n$. Then
\begin{eqnarray*}
\det \bigl[\bigl\{\langle v_i,v_j\rangle_n
\bigr\}_{i,j} \bigr]&\le &\prod_{i=1}^d
\langle v_i,v_i \rangle_{n}
\\
&\le&\prod_1^{d-M}\mathbf{c}
\varepsilon^2 \langle v_i,v_i
\rangle_{4 n} \prod_{i=d-M+1}^{d}
\langle v_i,v_i \rangle_{4n}
\\
&=&\bigl(\mathbf{c}\varepsilon^2\bigr)^{d-M}\det \bigl[\bigl
\{\langle v_i,v_j\rangle _{4 n}\bigr
\}_{i,j} \bigr],
\end{eqnarray*}
where in the first line we have used Hadamard's
inequality, in the second Lemma~\ref{lemma b} and in the last, the
fact that $(v_i)$ is orthogonal for
$\langle\cdot,\cdot\rangle_{4n}$. Now, the ratio of two Gram
determinants is preserved by linear operations on vectors, so we can
return from the basis $\{v_i\}$ (which was specific to $n$) to our
``original'' basis~$\{u_i\}$. We get
\[
\det \bigl[\bigl\{\langle u_i,u_j\rangle_n
\bigr\}_{i,j} \bigr]\le\bigl(\mathbf{c}\varepsilon ^2
\bigr)^{d-M}\det \bigl[\bigl\{\langle u_i,u_j
\rangle_{4n}\bigr\}_{i,j} \bigr].
\]
Iterating the reasoning, we find for every $r>0$
\[
\det \bigl[\bigl\{
\langle u_i,u_j\rangle_n\bigr
\}_{i,j} \bigr]\le \bigl[\bigl(\mathbf{c}\varepsilon ^2
\bigr)^{d-M} \bigr]^r\det \bigl[\bigl\{\langle
u_i,u_j\rangle_{4^rn}\bigr\}_{i,j}
\bigr].
\]
%
%Since every entry of the matrix is smaller than
%$\left[\Cc(4^rn)^k\right]^2\nu(\Bb_\rho(4^rn)\le\Cc(4^rn)^{2k+d}$
%thanks to the bound on the growth of the harmonic functions, we find
The growth of our harmonic functions ensures that every entry of the
matrix is smaller than
$ [\Cc(4^rn)^k ]^2\nu(\Bb_\rho(4^rn))\le
\Cc(4^rn)^{2k+\mathbf{c}}$. Hence we can write
\[
\det \bigl[\bigl\{\langle u_i,u_j\rangle_n
\bigr\}_{i,j} \bigr]\le d!n^{(2k+\mathbf{c})d}\Cc ^{d}\bigl(
\bigl(\mathbf{c}_1\varepsilon^2\bigr)^{d-M} 4
^{(2k+\mathbf{c}_2)d}\bigr)^r.
\]
We now fix $\varepsilon^2$ to be $4^{-4k-2\mathbf{c}_2}/\mathbf
{c}_1$. If $d>2M$, this implies
\[
\bigl(\mathbf{c}_1\varepsilon^2\bigr)^{d-M}4^{(2k+\mathbf
{c}_2)d}=4^{(2k+\mathbf{c}_2)(2M-d)}<1,
\]
and %and $\e=(\cc^{1+kd/(d-M_\e)})^{-1}$,
the right-hand side would
converges to 0. We deduce that $\det [\{\langle u_i,\break u_j\rangle_n\}
_{i,j} ]=0$, and that the $u_i$ restricted to the ball of radius $n$
form a dependent family. Since this is true for every $n$ large enough,
we easily deduce that $(u_i)$ is a linearly dependent family. The
result holds for any family of $d$ harmonic functions with growth
bounded by $\Cc\Dd(\cdot,\rho)^k$. It implies that the dimension of the
vector space of harmonic functions with such growth is smaller or equal
to $2M$.
\end{pf*}

%\begin{remark}
%The event that the dimension of the space of harmonic functions with
%some prescribed growth is equal to $d$, is invariant under shift by
%$X_1$. Therefore, for any ergodic stationary random graph, it must be
%constant. \gady{How is Liouville's property relevant here? I did not
%understand. (Hugo) to have the 0-1 law on translation-invariant
%events?} Since stationary random graphs with polynomial growth have
%trivial tail $\sigma$-algebra \cite{BC10} and are therefore trivial.
%In particular, the dimension must be constant.
%\end{remark}

%ex3.1 #&#
\begin{example}[(Infinite cluster of percolation)]
The infinite cluster of percolation satisfies
$(\mathit{VD})_\omega$ and $(P)_{\omega}$ almost
surely \cite{Bar04}. Therefore, spaces of harmonic functions with
prescribed polynomial growth are finite dimensional.
\end{example}

%ex3.2 #&#
\begin{example}[(Random conductance)]
Random conductances with uniform elliptic conditions also satisfy
$(\mathit{VD})_\omega$ and $(P)_{\omega}$ deterministically. Therefore, spaces of harmonic
functions with prescribed polynomial growth are finite dimensional.
\end{example}

%ex3.3 #&#
\begin{example}[(Wedges)]Let $f$ be some \emph{slowly varying}
function from $[0,\infty)\to[0,\infty)$. Define the wedge with
respect to $d$ and $f$ to be
\[
W:=\bigl\{x\in\Z^d\dvtx |x_d|\le f\bigl(|x_1|+
\dotsb+|x_{d-1}|\bigr)\bigr\}.
\]
Then it is well-known and not difficult to see that $W$ (with the
graph structure inherited from $\Z^d$) satisfies volume doubling and
Poincar\'e inequality. Under some weak conditions on $f$ and $d$
(which we will not detail here, as that would take us too off-topic)
so would percolation on $W$. Hence both $W$ and supercritical
percolation on it have a finite dimensional space of harmonic functions.
\end{example}

\section{Linearly growing harmonic functions on the infinite cluster of
percolation}\label{sec:linear growth}

In this section, we fix $d \geq2$ and $p>p_c(d)$. As before, we denote
the infinite cluster of percolation by $\omega$, and we draw it in
$\mathbb R^d$ in such a way that $\rho$ coincides with the origin. The
graph $\omega$ can be thought of as an approximation of~$\Z^d$. In
particular, macroscopic properties of the cluster are the same as those
of~$\mathbb R^d$. For instance, the random walk satisfies an invariance
principle $(\mathit{CLT})_\omega$ \cite{BB07,MP07,SS04}: define
\[
\tilde B_n(t):= \frac{1}{\sqrt n} (X_{tn} ),
\]
where for noninteger $tn$ we define $X_{tn}$ as the linear
interpolation between $X_{\lfloor tn \rfloor}$
and $X_{\lceil tn \rceil}$;
that is,
$X_{tn} = X_{\lfloor tn \rfloor} (tn - \lfloor tn \rfloor) +
X_{\lceil
tn \rceil} (\lceil tn \rceil- tn) $.
%(where we slightly abuse notations by defining $X_t$ for $t$
%nonintegers by approximating linearly between $\lfloor t \rfloor$ and
%$\lceil t \rceil$).
There exists $\sigma(d)$ such that the law of $(\tilde
B_n(t),0<t<\infty)$ converges weakly to the law of a Brownian motion
with variance $\sigma(d)$ as $n\to\infty$. The main step in the proof
in all three papers
\cite{BB07,MP07,SS04} is the construction of a $d$-dimensional space of
linearly growing harmonic functions $ \{f_v \}_{v\in\R^d}$
such that $f_v$ has
slope $v$, that is, $f_v(x)=\langle v,x\rangle+o(|x|)$. Let us state
this as a theorem.

\begin{theorem}[(\cite{BB07,MP07,SS04})]\label{corrector}
Let $d\geq2$, and $p>p_c(d)$.
Let $\omega$ be the infinite cluster of percolation on $\Z^d$ with
parameter $p$.
Then,
there exists $\chi\dvtx \omega\rightarrow\mathbb R^d$ such that $x\mapsto
x+\chi(x)$ is harmonic on $\omega$, and
\begin{equation}
\lim_{n\rightarrow\infty}\frac{1}n\sup_{x\in\Bb
_\rho
(n)} \bigl|
\chi(x) \bigr|=0\qquad \mbox{a.s.}
\end{equation}
This (random) function is called the \emph{corrector}.
\end{theorem}

With the constant
functions, we get a $(d+1)$-dimensional space of harmonic functions with
\mbox{(sub-)}linear growth.
Our aim in this section is to prove Theorem~\ref{dimension linear} from
the introduction,
namely that there are no other harmonic functions of linear growth.

\subsection*{Proof outline}\label{pg:sketch} Let $h$ be a harmonic
function with
linear growth.
%We extend to $\mathbb R^d$ any function defined on $\omega$ in such a
%way that the function is piecewise linear on faces of the square
%lattice.
Define $h_n\dvtx \mathbb R^d\rightarrow\mathbb R$ such that
$h_n(x)=h(nx)/n$. In order to prove Theorem~\ref{dimension linear}, we
first show that $(h_n)$ forms a precompact family (one can say that $h$
has a scaling limit). The second step is to identify the possible
limits. For this, we use the average property at the discrete level and
the invariance principle to prove that limits are harmonic on $\mathbb
R^d$. If the space of limits is at most $d$-dimensional, one can then
use the absence of nonconstant sublinear harmonic functions to show
that the space of harmonic functions with linear growth is $(d+1)$-dimensional.

\subsection*{Properties of the supercritical cluster} Recall
that the infinite supercritical cluster of percolation $\omega$ can be
seen as a
stationary random graph with polynomial growth. It is well-known that
the system is ergodic with respect to the shift by~$X_1$, see, for example, \cite{BB07}, Theorem~3.1. Typical balls have
the same
growth as in the ambient space $\Z^d$ in the following sense: there
exists constants $c$ and $C$ such that for any finite $\lambda$, any
$n > n_0(\omega)$ sufficiently large and any $x\in\Bb_\rho(\lambda n)$
\begin{equation}
\label{eq:vold} c n^d\le\nu\bigl(\Bb_x(n)\bigr)\le
Cn^d.
\end{equation}
Clearly, \eqref{eq:vold} implies volume doubling \VdW{}. Moreover, the
graph satisfies \PW{} almost surely. Both properties were proved by
Barlow \cite{Bar04}. Actually, Barlow proved quantitatively stronger
versions of \eqref{eq:vold}
and \PW{}: he obtained the volume growth estimates and
the Poincar\'e inequality for every ball of radius larger than $C\log
n$ in $\Bb_\rho(n)$. These improved versions allow to prove Harnack
inequalities and Gaussian estimates (\ref{a}) on the heat kernel. In
\cite{Bar04}, these results are stated for
continuous time random walk, but they hold also for simple random
walk, as was explained in \cite{BH09}, Section~2. We do not need the
full force of Gaussian estimates here---in particular
we do not need far off-diagonal lower bounds which are particularly
difficult---so let us make a list of corollaries from these
Gaussian estimates which we will use.

\begin{corollary}
For every $\lambda<\infty$, every $n\in\mathbb N, n>n_0(\omega)$
sufficiently large and every $x\in\Bb_\rho(\lambda n)$,
\begin{equation}
\label{eq:upbound} \Pp_x(X_{n^2}=y)\le C n^{-d}\exp
\bigl[-C\Dd(x,y)^2/n^2 \bigr] \qquad\mbox{for any }y.
\end{equation}
\end{corollary}

In both \cite{Bar04,BH09} the results are formulated with $|x-y|$
instead of $\Dd(x,y)$, but by the results of Antal and Pisztora
\cite{AP96}, this is the same.
This immediately implies
\begin{equation}
\Ee_x\bigl[\Dd(X_{n^2},x)^2\bigr]\le
\mathbf{c}_3 n^2\label{quenched diff}
\end{equation}
for some constant $\mathbf{c}_3$ depending on the environment.

The lower bound has some periodicity requirements since
$\Pp_x(X_t=y)=0$ whenever $t + \sum(x_i-y_i)$ is odd.

\begin{corollary}
For every $\lambda<\infty$, every $n\in\mathbb N$ sufficiently large
and every
$x\in B(\lambda n)$,
\[
\label{eq:lowbound}
\Pp_x(X_{n^2}=y)\ge C n^{-d}
\]
for any $y\in\Bb_x(n)$ such that $n^2+\sum
(x_i-y_i)$ is even.
\end{corollary}

%Let us recall the bounds we will be using,
%for the simple random walks: for every $\lambda>0$, for every $x,y\in
%\Bb_\rho(\lambda n)$ and $n$ large enough,
%\begin{equation}\label{az}\frac{C_1}{n^{d/2}}
%\exp\big[-C_2|x-y|^2/n\big]
%1_{\{\Pp_x[X_n=y]>0\}}~\le~\Pp_x[X_n=y]~\le~\frac{C_3}{n^{d/2}}~\exp
%\big[-C_4|x-y|^2/n\big].\end{equation}
%This immediately implies
%We start with a technical lemma which allow to turn estimates at
%$\rho$ into estimates at most points. The main ingredient is
%ergodicity.
%which asserts that $\Delta_n(x,y)$
%(for $x\sim y$) cannot be too large in average on the space (recall
%that $\Delta_n$ is our strengthened version of the total variation
%distance between the distributions of random walks starting from $x$
%and $y$, (\ref{eq:defDeltan})). Note that we already know that it is
%not too large in average on the environment, since the $L^2$-norm is
%bounded.

In the proof we will need in a few places \emph{space ergodicity}. We
start with a lemma that encapsulates this for us.

\begin{lemma}\label{good average}
Let $f(x,y,\omega)$ be some positive translation-invariant variable,
that is, $f(x+s,y+s,\omega+s)=f(x,y,\omega)$, with
$M:=\E f(0,X_1,\omega)<\infty$. Then for every $\lambda>0$
and for almost every environment $\omega$,
there exists $n_0$ such that for all $n>n_0$ and for any $a\in\Bb
_\rho
(\lambda n)$,
\[
\sum_{(x,y)\in E(\Bb_a(
n))}f(x,y,\omega) \nu(x,y)\leq C \cdot M \cdot
n^{d}.
\]
\end{lemma}

%Before starting the proof, observe the following fact. Since the
%infinite cluster of percolation is a subgraph of $\Z^d$, it has
%uniform polynomial growth and $H_n\le C_1\log n$ for every $n$.
%Theorem~\ref{entropy} implies that $\mathbb E[\Delta_n(\rho,X_1)^2]\le
%C_2/n$ for an infinite number of $n$. Using Fatou's lemma, we obtain
%that $\mathbb E[\Delta_\infty(\rho,X_1)^2]<\infty$.

\begin{pf}
We wish to apply the ergodic theorem for
$\Z^d$ actions (see, e.g., \cite{S09}, Theorem~2.6, page 40). We let
the probability space be $\{0,1\}^{\Z^d}$ with the product measure and
let the probability preserving maps $T_i$ from the statement of the
theorem to be translations of coordinates. Clearly the $T_i$ commute.
Further, each $T_i$ has only trivial invariant subsets---indeed, if $A$ is invariant under some $T_i$, then we can $\varepsilon
$-approximate $A$
by an event $B$ depending only on finitely many coordinates and then
apply $T_i^n$ for $n$ sufficiently large so that $B$ and $T_i^nB$ are
independent. We get that $|\P(A)-\P(A)^2|\le3\varepsilon$. Since
$\varepsilon$ is
arbitrary, $\P(A)$ must be equal to $0$ or $1$.

Fix $v$ to be one of the $d$ vectors of the standard basis of $\Z^d$
and define a function $F\dvtx \{0,1\}^{\Z^d}\to\R$ by
\[
F(\xi)= %
\cases{f\bigl(\rho,v,\omega(\xi)\bigr), &\quad $\rho,v\in\omega(
\xi),$\vspace*{2pt}
\cr
0, &\quad $\mbox{otherwise},$} %
\]
where $\omega(\xi)$ is the infinite cluster (possibly equal to the
empty set). The ergodic theorem implies
%\[
%\lim_{n\to\infty}\frac1{(2n)^d}\sum_{0\le i_1,\ldots,i_d\le2n}
%F(T_1^{i_1}\cdots T_d^{i_d}\xi)= \E(F) \mbox{almost surely}
%\]
%(the theorem in \cite{S09} is formulated for $\Z_+^d$ actions).
%Divide
%$[-n-1,n]^d$ into $2^d$ copies of $[0,n]^d$ and using the ergodic
%theorem $2^d$ times for all choices of $\pm T_i$ (and to fixed
%translations of $x$) gives the needed $\Z^d$ version:
\[
\lim_{n\to\infty}\frac{1}{(2n)^d}\sum_{-n-1\le i_1,\ldots,i_d\le n}
F\bigl(T_1^{i_1}\cdots T_d^{i_d}\xi
\bigr)= \E(F) \qquad\mbox{almost surely}
\]
(the theorem in \cite{S09} is formulated for $\Z_+^d$ actions, but the
two-sided version above follows from the one-sided version by applying
the one-sided result $2^d$ times, for each choice of $T_1^{\pm
1},\ldots,T_d^{\pm1}$,
and combining the results).
Summing the above over all $v$ in the standard basis element $v$
enables us to go from $F$ to $f$ and to obtain
\[
\lim_{n\to\infty}\frac{1}{|E(Q_n)|}\sum_{(x,y)\in
E(Q_n))}f(x,y,
\omega )\le CM,
\]
where $Q_n=[-n-1,n]^d$ and $C$ is some universal constant.

Let us now generalize this to cubes centered around an arbitrary $a\in
B_\rho(\lambda n)$. Fix some $N,C'$ sufficiently large such that
\[
\P \biggl(\exists n>N, \sum_{(x,y)\in E(Q_n)}f(x,y,
\omega)>C'M n^d \biggr)<\mu,
\]
where $\mu$ will be defined in the next paragraph (as a function of
$\lambda$). Define
$\omega$ to be \emph{good} if the event involved in the previous
displayed equation does not happen, and define $b\in\Z^d$
to be \emph{good} if translating $\omega$ by $-b$ gives a good
configuration. Using ergodicity once again,
there exists almost surely $n_0=n_0(\omega)<\infty$ such that for any
$n\ge n_0$,
\[
\bigl|\{\mathrm{good}\ b\in Q_n\}\bigr|>(1-2\mu)|Q_n|.
\]
In particular, the number of sites in $\Bb_\rho(\lambda n)\subset
Q_{\lambda n}$ which are not good is less than $2\mu|Q_{\lambda n}|$
for $n\ge n_0$. If $\mu=\mu(\lambda)$ is chosen sufficiently small and
$n\ge n_0$, then there cannot be any good-free ball of
radius $n$ in $\Bb_\rho(\lambda n)$. Hence, for any $a\in
\Bb_\rho(\lambda n)$, there is a cube $b+Q_{2n}\supset\Bb_a(n)$ centered
around a good point $b$. This implies that for $n\ge\max\{n_0,N\}$,
\[
\sum_{(x,y)\in E(\Bb_a(n))}f(x,y,\omega)\le\sum
_{(x,y)\in
E(b+Q_{2n})}f(x,y,\omega) \le C'M\cdot(2n)^d.
\]
(The assumption $n\ge N$ enables us to use the fact that $b$ is good.)
Adding the terms $\nu(x,y)$ only changes the constant.
\end{pf}

Recall the $h_n$ from the proof sketch on page
\pageref{pg:sketch}. There we defined $h_n(x)=h(nx)/n$ which is
{a priori} only defined on the contracted infinite cluster. For
simplicity let us extend it to all $\R^d$, for example, by extending $h$ to
$\Z^d$ by taking the value at the closest point of the
infinite cluster, and then to $\R^d$ by defining
$h(x)=\sum_{y\in\Z^d}h(y)\phi_y(x)$ where $\phi_y$ is some partition
of unity such that $\supp\phi_y\subset y+[-\frac{2}3,\frac{2}3]^d$.
%dividing each square
%$(n,m)+[0,1]^2$ into two triangles and interpolating on the triangles
%linearly.
Once $h$ is extended to all $\R^d$, so is $h_n$.

\begin{proposition}\label{continuity}
For almost every environment $\omega$, any harmonic function $h$ on
$\omega$ with linear growth satisfies that for every compact $K\subset
\mathbb R^d$, the sequence $(h_n)|_{K}$ is uniformly bounded and
equicontinuous. %(where $h_n=h(n\cdot)/n$).
\end{proposition}

\begin{pf}
Fix a harmonic map $h$ with (at most) linear growth on an environment
$\omega$. There exists $A>0$
such that $|h(x)|\le A|x|$. We only need to prove equicontinuity on
the ball, as this property passes to subsets.
% (other compact sets $K$ work the same).
To do so, we prove that for any $\eta>0$, there
exists $\delta>0$ such that $(h(a)-h(b))^2\le\eta n^2$ for any two
points $a,b\in\Bb_\rho(n)$ at distance $\delta n$ of each other, when
$n$ is large enough (it is easy to see that our procedure for
extending $h$ to $\R^d$ allows to prove the needed estimates only for
$a$ and $b$ in $\omega$). For this reason, we will always assume that $n$
is large enough so that the Poincar\'e inequality \PW{} and the
$d$-dimensional volume growth \eqref{eq:vold} hold true for $\lambda=2$.

Let $\delta,\varepsilon>0$ to be fixed later (think of
$\varepsilon\ll\delta$) and $a,b\in
\Bb_\rho(n)$ with $\Dd(a,b)\le\delta n$. Let $B$ be some ball of
radius $2\delta n$ containing both $\Bb_a(\delta n)$ and $\Bb
_b(\delta
n)$---for example, around the middle point of $[ab]$. Let $\overline
{h}$ be the
average $\frac{1}{\nu(B)}\sum_{x\in B}h(x)\nu(x)$. Since
$|h(a)-h(b)|\le|h(a)-\overline{h}| + |h(b)-\overline{h}|$, it is
enough to estimate these terms. Let us focus on estimating
$|h(a)-\overline{h}|$ (the other term is symmetric).

Set $\mathcal E$ to be the event that $|X_{(\varepsilon n)^2}-a|\ge
\delta n$.
Note that
\[
\mathbf{P}_a(\mathcal E) \le\frac{\Ee_a (|X_{(\varepsilon n)^2}-a|^2 )}{(\delta n)^2} \stackrel{
\mathrm{\scriptsize (\ref{quenched diff})}} {\le} \frac{\mathbf{c}_3(\varepsilon n)^2}{(\delta n)^2} =\mathbf{c}_3 (
\varepsilon/\delta)^2,
\]
where the Markov inequality was used in the first inequality and the
quenched diffusive behavior \eqref{quenched diff} in the second.

Now, we have
\begin{eqnarray*}
\bigl|h(a)-\overline{h}\bigr|^2&\le& \bigl(\Ee_a
\bigl[\bigl|h(X_{(\varepsilon
n)^2})-\overline {h}\bigr| \bigr] \bigr)^2
\\
&\le&2 \bigl(\Ee_a \bigl[\bigl|h(X_{(\varepsilon n)^2})-\overline {h}\bigr|{\mathbf
1}_{\mathcal E} \bigr] \bigr)^2+ 2 \bigl(\Ee_a
\bigl[\bigl|h(X_{(\varepsilon n)^2})-\overline{h}\bigr|{\mathbf 1}_{\mathcal
E^c} \bigr]
\bigr)^2.
\end{eqnarray*}
We first deal with the first term on the right:
\begin{eqnarray*}
\bigl(\Ee_a \bigl[\bigl|h(X_{(\varepsilon n)^2})-\overline{h}\bigr|{\mathbf
1}_{\mathcal
E} \bigr] \bigr)^2 &\le&\Ee_a \bigl[
\bigl(\bigl|h(X_{(\varepsilon n)^2})\bigr|+|\overline{h}| \bigr)^2 \bigr]\cdot
\mathbf{P}_a(\mathcal E)
\\
&\le& \bigl( 2\Ee_a \bigl[h(X_{(\varepsilon n)^2})^2 \bigr]+2
\overline{h}^2 \bigr)\cdot{\mathbf P}_a(\mathcal E)
\\
\mbox{since $h(x)\le A|x|$}\qquad &\le&2A \bigl(\Ee_a
\bigl[|X_{(\varepsilon n)^2}|^2 \bigr]+(1+2\delta )^2n^2
\bigr) \cdot\mathbf{P}_a(\mathcal E)
\\
\mbox{by \eqref{quenched diff}}\qquad &\le&2A\bigl(\mathbf{c}_3\bigl(1+
\varepsilon^2\bigr)+(1+2\delta)^2\bigr)n^2
\cdot{\mathbf c}_3\frac{\varepsilon^2}{\delta^2} =\mathbf{c}_5n^2
\frac{\varepsilon^2}{\delta^2},
\end{eqnarray*}
where Cauchy--Schwarz was used in the first inequality.%, in the third
%the bound $h(x)\le A|x|$, and in the last, \eqref{az}.

For the second term, the heat kernel upper bound \eqref{eq:upbound}
shows that\break
$\Pp_a(X_{(\varepsilon n)^2}=x)\le C_6/(\varepsilon n)^{d}$ for any
$x\in\Bb_\rho(n)$
and $n$ large enough. Therefore,
\begin{eqnarray*}
\bigl(\Ee_a \bigl[\bigl|h(X_{(\varepsilon n)^2})-\overline{h}\bigr|{\mathbf
1}_{\mathcal
E^c} \bigr] \bigr)^2&\le& \Ee_a
\bigl[\bigl|h(X_{(\varepsilon n)^2})-\overline{h}\bigr|^2{\mathbf 1}_{\mathcal
E^c}
\bigr]
\\
%\frac{\cc_3}{(\e n^2)^{d}}\int\int_{\Bb_a(\delta n)\times\Bb_b(\delta
%n)}|h(x)-h(t)|^2d\mu_xd\mu_t\\
&\le& %\frac{\delta^{d/2}}{\e^{d/2}}\frac{\cc_3}{n^{d}}
\frac{C_6}{(\varepsilon n)^{d}} \sum
_{x\in\Bb_a(\delta n)}\bigl\llvert h(x)-\overline{h}\bigr\rrvert ^2
\nu (x)
\\
&\le& %\frac{\delta^{d/2}}{\e^{d/2}}\frac{\cc_3}{n^{d}}
\frac{C_6}{(\varepsilon n)^{d}} \sum_{x\in B}\bigl
\llvert h(x)-\overline{h}\bigr\rrvert ^2 \nu(x).
\end{eqnarray*}
%
%where we have used \VdW to in the second inequality.
Poincar\'e's inequality implies
\[
\Ee_a \bigl[\bigl|h(X_{(\varepsilon n)^2})-\overline{h}\bigr|{\mathbf
1}_{\mathcal
E^c} \bigr]^2\le\frac{\CP C_6}{(\varepsilon n)^{d}}(2\delta
n)^2 \sum_{(x,y)\in E(B')}\bigl|h(x)-h(y)\bigr|^2
\nu(x,y),
\]
where $B'$ is the ball with same center as $B$ and radius $4\delta n$.

Now, the quantity $\Delta_n$ introduced in \eqref{main inequality}
controls the gradient of a harmonic function. Indeed, the same
reasoning as the one used to derive \eqref{crucial inequality} implies that
\[
\bigl|h(x)-h(y)\bigr|^2\le \bigl(\Ee_x\bigl[\bigl|h(X_n)\bigr|^2
\bigr]+\Ee_y\bigl[\bigl|h(X_{n-1})\bigr|^2\bigr] \bigr)
\Delta_n(x,y)^2%\\
%&\le A^2\Big(\Ee_x[|X_n|^2]+\Ee_y[|X_{n-1}|^2]\Big)\Delta_n(x,y)^2
\]
for every $n$. Using the bound $|h(z)|\le A|z|$, diffusivity and taking
the liminf, we obtain
\[
\bigl|h(x)-h(y)\bigr|^2\le\mathbf{c}_7 \liminf
_{n\to\infty} n \Delta_n(x,y)^2,
\]
where $\mathbf{c}_7$ does not depend on the points $x,y$ (though it does depend
on $h$ through~$A$). Denote this $\liminf$ by
$\Delta_\infty(x,y)^2$. We get
\begin{eqnarray*}
\bigl(\Ee_a \bigl[\bigl|h(X_{(\varepsilon n)^2})-\overline{h}\bigr|{\mathbf
1}_{\mathcal
E^c} \bigr] \bigr)^2\le \frac{\delta^{2}}{\varepsilon^{d}}
\frac{\mathbf{c}_8}{n^{d-2}}\sum_{(x,y)\in
E(B)}\Delta_\infty(x,y)^2
\nu(x,y).
\end{eqnarray*}
We next note that $\E\Delta_\infty(\rho,X_1)^2<\infty$. Indeed, the
infinite cluster of percolation is a subgraph of $\Z^d$, it has
uniform polynomial growth and $H_n\le C_1\log n$ for every
$n$. Theorem~\ref{entropy} implies that $\mathbb
E[\Delta_n(\rho,X_1)^2]\le C_2/n$ for an infinite number of $n$. Using
Fatou's lemma, we obtain that $\mathbb
E[\Delta_\infty(\rho,X_1)^2]<\infty$. Thus we may use Lemma~\ref{good
average} for the function $f=\Delta_\infty^2$ and get (with the fact
that $B'$ has radius $4\delta n$),
\[
\Ee_a \bigl[\bigl |h(X_{(\varepsilon n)^2})-\overline {h} \bigr|{\mathbf
1}_{\mathcal E^c} \bigr]^2\le\mathbf{c}_{9}
\frac{\delta
^{d+2}}{\varepsilon^{d}}n^2.
\]
Putting together the estimates for the two terms, we obtain
\[
\bigl(h(a)-h(b) \bigr)^2\le n^2 \biggl(
\mathbf{c}_3\frac{\varepsilon
^2}{\delta^2}+\mathbf{c} _{9}
\frac{\delta^{d+2}}{\varepsilon^{d}} \biggr),
\]
which implies the claim provided $\delta=\varepsilon^{(d+1)/(d+2)}$.
\end{pf}

\begin{lemma}\label{lem:subseq_linear}For almost every environment
$\omega$, for any harmonic function $h$ on $\omega$ with linear
growth, any
subsequential limit of $h_n$ is linear.
\end{lemma}

\begin{pf}Let $n_k$ be a sequence such that $h_{n_k}$ converges
uniformly on compact subsets of $\mathbb R^d$, and denote the limit
by $\ell$. Let now $B_t$ be a Brownian motion with variance $\sigma
(d)$, where
$\sigma(d)$ comes from the invariance principle for random walk on
$\omega$, see page \pageref{sec:linear growth}. Our first goal is to
derive a mean-value property anchored at the origin. Namely, we wish to
prove that
\begin{equation}
\label{eq:from0} \Ee_0 \bigl[\ell(B_t)\bigr]=\ell(0)\qquad
\mbox{for any }t>0.
\end{equation}
To see (\ref{eq:from0}) note that $h$ is harmonic and hence
$\Ee_\rho[h(X_t)]=h(\rho)$ or equivalently
\[
\Ee_0\bigl[h_{n}(X_{n^2t}/n)
\bigr]=h_{n}(0).
\]
The central limit theorem (Theorem~\ref{corrector}) allows to control
$h(X_t)$ in a ball of radius
$\approx\sqrt{t}$. Namely, because $X_{n^2t}/n$ converges weakly to
$B_t$, and because $\ell$ is continuous (as a locally uniform limit of
the $h_{n_k}$), for any $K>0$,
\[
\bigl\llvert \Ee_0 \bigl[\ell(X_{n^2t}/n)\cdot{
\mathbf1}_{\{|X_{n^2
t}/n|<K\}
} \bigr]- \Ee_0 \bigl[\ell(B_t)
\cdot{\mathbf1}_{\{|B_t|<K\}} \bigr]\bigr\rrvert \to 0,
\]
where the convergence
is as $n\to\infty$. The Gaussian bounds (\ref{eq:upbound}) and the linear
bounds on $h_{n}$ and $\ell$ allow to control $h(X_t)$ outside that ball,
\[
\bigl\llvert \Ee_0 \bigl[h_n(X_{n^2t}/n)\cdot{
\mathbf1}_{\{|X_{n^2t}/n|\ge
K\}} \bigr]\bigr\rrvert \le\varepsilon(K)\qquad \mbox{for any }n
\mbox{ sufficiently large},
\]
where $\varepsilon(K)\to0$ as $K\to\infty$. A similar estimate
holds for $\ell
(B_t)$. This shows (\ref{eq:from0}).

We now extend (\ref{eq:from0}) from 0 to all points $u$ using Lemma~\ref{good average}. Let us recall that weak convergence is
metrizable. For example, it is equivalent to convergence in the
L\'evy--Prokhorov distance metric (e.g., \cite{D02}, Section~11.3), especially
Theorem~11.3.3. We will not need any property of the
L\'evy--Prokhorov metric except that it is equivalent to weak
convergence.

Now fix $t$, and fix also some $\varepsilon$ and some
$n_0$. Consider a vertex $x$ in the cluster to be good if the Gaussian
estimates (\ref{eq:upbound}) hold for all $n>n_0$ and if the
L\'evy--Prokhorov distance between
$X_{n^2t}/n$ (started from $x$) and $B_t$ (started from $x/n$) is
smaller than $\varepsilon$, again for all $n>n_0$. If $n_0$ is sufficiently
large (depending on $t$ and $\varepsilon$), the
probability of $x$ being good will be larger than $1-\varepsilon$. For
the Gaussian
estimates this follows directly from (\ref{eq:upbound}) while for the
L\'evy--Prokhorov distance this follows from the equivalence of
L\'evy--Prokhorov convergence and weak convergence. Fix therefore $n_0$
to satisfy this property.

Now use Lemma~\ref{good average} with the function $f$ being
$f(x,y)={\mathbf1}_{\{x\,\mathrm{is\,bad}\}}$ (the $y$ variable is simply
ignored) and with some arbitrary $\lambda$. We get that for
sufficiently large~$n$, the number of bad $x$ in $\Bb_\rho(\lambda n)$
is bounded by $C(\lambda)n^d\P(0$ is bad$)\le C(\lambda
)\varepsilon n^d$. Define
\[
B_n:=\bigl\{u\in\R^d\dvtx |u|\le\lambda, un\mbox{ is
bad}\bigr\}
\]
(where as usual we in fact take the point of the infinite cluster
closest to $un$ and check whether it is bad).
Since the measure of $B_n$ is smaller than $ C\varepsilon$, we see
that, except
for a set
of measure smaller than $ C\varepsilon$, every $u\in\R^d$ with
$|u|\le\lambda$ satisfies
that $un_k$ is good for infinitely many $n_k$ (it does not matter that
$n_k$ itself depends on the environment here). But $\varepsilon$
(both for the error and for the measure of the bad set) was
arbitrary. Taking $\varepsilon\to0$ and then $\lambda\to\infty$ we see
that for almost every $u\in\R^d$ there is a sequence $ n_k'=n_k'(u)$ (a
subsequence of $n_k$) such that:
\begin{longlist}[1.]
\item[1.] The Gaussian estimates hold for $X_{n_k'}$ started from $un_k'$.
\item[2.] The L\'evy--Prokhorov distance between $X_{(n_k')^2 t}/n'_k$
started from $un'_k$
and Brownian motion started from $u$ goes to zero.
\end{longlist}
Using again the equivalence of L\'evy--Prokhorov convergence and weak
convergence we get that random walk started from $un'_k$ converges to
Brownian motion started from $u$. We can now repeat the argument that
led to \eqref{eq:from0} literally and get
\[
\Ee_u\bigl[\ell(B_t)\bigr]=\ell(u)
\]
for almost every $u$. Since $\ell$ is continuous, this in fact holds
everywhere. Since $t$ was arbitrary, $\ell(B_t)$ is a continuous
martingale, from any starting point.

The lemma is now proved. Using the strong Markov property we get
that $\ell(u)$ is equal to its average over a sphere of arbitrary
radius around $u$, in other words, we have established the mean-value
property hence $\ell$ is (continuously) harmonic and has a linear
bound. It is well known that harmonic functions with at most linear
growth on $\mathbb R^d$ are the affine maps (take the partial
derivative along one direction, it is a bounded harmonic map on
$\mathbb R^d$, and thus a constant map).
\end{pf}

\begin{pf*}{Proof of Theorem~\ref{dimension linear}}
Let $d\ge2$. The constant functions on $\omega$ are obviously
harmonic. The projections of $x+\chi(x)$ where $\chi$ is the corrector
(see Theorem~\ref{corrector}) on each coordinate provide us with $d$
linearly independent functions. These functions have linear
growth. Therefore, the space of linear growth harmonic functions is at
least $(d+1)$-dimensional. Thus we need to show that any harmonic
function of linear growth is of the form $h(x+\chi(x))$.

The first step is to apply Theorem~\ref{entropy} and get that $\mathbb
E(\Delta_n(\rho,X_1)^2)\le2(H_n-H_{n-1})$ and in particular is $\le
C/n$ on a subsequence. Hence, by Fatou's lemma, there is a
random subsequence $n_k$ such that $\Ee_\rho[\Delta_{n_k}(\rho
,X_1)^2]\le C/n_k$.

Now, let $h$ be a harmonic function on $\omega$ with (at most) linear
growth and with $h(0)=0$. Proposition~\ref{continuity} allows us to
extract a sequence $m_k$ such that $(h_{m_k})$ converges uniformly on any
compact subset of $\mathbb R^d$ to a continuous function $\tilde
h$, and further one may take $m_k$ to be a subsequence of any given
sequence, so we may assume $m_k$ is a subsequence of $\lfloor
n_{k}^{1/2}\rfloor$. %For simplicity, we forget about the subsequence
%$(n_k)$ and assume
%that the sequence is converging.
By Lemma~\ref{lem:subseq_linear} $\tilde h$ is linear. We get that,
$f(x):=h(x)-\tilde h(x+\chi(x))$ is a harmonic function on $\omega$
with the
following additional property: for every $\varepsilon>0$ there exists
$k_0\in
\mathbb N$ such that for all $k>k_0$,
\[
\bigl|f(x)\bigr|\le\varepsilon m_k \qquad\mbox{for any }x\in\omega\ \mathrm{with}\
\Dd (x,\rho)<\frac{1}{\varepsilon}m_k.
\]
that is, it is sublinear on a sequence of (space) scales. A simple
calculation with the Gaussian upper bounds \eqref{eq:upbound} and the
fact that $h$ has a linear bound shows
that it is also sublinear on a sequence of time scales, that is,
\begin{equation}
\label{eq:mk2 not an ad} \Ee_\rho \bigl[f(X_n)^2 \bigr]
\le\varepsilon n\qquad \forall n\in\bigl[\tfrac12 m_k^2,2m_k^2
\bigr], k>k_0'.
\end{equation}
Since the $m_k$ were approximate square roots of a subsequence of the
$n_k$, we may find a subsequence $n_k'$ of $n_k$ for which
$\Ee[f(X_{n_k'})^2]\le\varepsilon n_k'$.

We now repeat the argument of Theorem~\ref{no sublinear harmonic}:
Equation \eqref{crucial inequality} still holds for every~$n$:
\[
\Ee_\rho \bigl|f(\rho)-f(X_1) \bigr|\le \sqrt{2\Ee_\rho
\bigl[\Delta_n(\rho,X_1)^2 \bigr]
\Ee_\rho\bigl[f^2(X_n)\bigr]}.
\]
For our $n_k'$ we have
$\Ee[\Delta^2]\le C/n_k'$, and with \eqref{eq:mk2 not an ad} we get
$\Ee_\rho|f(\rho)-f(X_1)|\le\sqrt{C\epsilon}$. Since $\epsilon$ was
arbitrary, $f$
%By Corollary~\ref{cor:subsequencescales}, it
must be constant. Since
$f(0)=0$ that constant is zero and $h=\tilde h(x+\chi(x))$. Therefore,
any harmonic function with growth at most linear and equal to $0$ at 0
belongs to a vector space of dimension $d$ and the result follows.
\end{pf*}

A natural extension of the supercritical bond percolation setting is
to look at random environments on $\Z^d$, such as the random
conductance model.
See \cite{BD10,BBHK08,BP07,SS04}
for the existence of the corrector in different cases of this model.
Similar results can probably be obtained in this setting.

\section{Heat kernel derivative estimates}\label{sec:heat kernel}

Our purpose in this section is to prove Theorem~\ref{heat kernel}
which gives an upper bound for the (discrete) derivative of the heat
kernel, $\mathbf{p}_n(x,y)-\mathbf{p}_{n-1}(x',y)$, for $x \sim x'$, where
$\mathbf{p}_n(x,y):=\Pp_x(X_n=y)$.
%Recall that we denote the lazy random walk

%on the graph by $\tilde X_n$ and that we denote.
%It is straightforward to see that $(G,\rho,\mu)$ is stationary if and
%only if
%$(G,\rho,\mu)$ and $(G,\tilde X_1,\mu)$ have the same law. In this
%section, we set

We start with a lemma true on any graph. It relates the infinity norm
of the gradient of the heat kernel to the infinity norm of the heat
kernel and the entropy.

\begin{lemma}
Let $G$ be a graph of maximal degree $d$. Then for any $x, x',y\in G$
with $x\sim x'$,
\begin{eqnarray}\qquad
\label{zz} &&\bigl(\mathbf{p}_{2n}(x,y)-\mathbf{p}_{2n-1}
\bigl(x',y\bigr) \bigr)^2
\nonumber
\\[-8pt]
\\[-8pt]
\nonumber
&&\qquad\le4d(d+1) \cdot\Delta_n\bigl(x,x'\bigr)^2
\cdot\mathop{\max_{a,b\in\Bb
_x(2n)\dvtx }}_{ \Dd(a,b) \geq\Dd(x,y)/2}\mathbf{p}_{n}(a,b)
\cdot\max_{a,b\in\Bb
_x(2n)}\mathbf{p}_n(a,b),
\end{eqnarray}
%
%where $\Vert Y\Vert _\infty$ is the infinity norm of $Y$.
where $\Delta_n$ is defined in \eqref{eq:defDeltan}.
\end{lemma}

\begin{pf}
Markov's property gives that
\[
\mathbf{p}_{2n}(x,y)-\mathbf{p}_{2n-1}\bigl(x',y
\bigr)=\sum_{a\in G}\bigl({\mathbf p}_n(x,a)-
\mathbf{p} _{n-1}\bigl(x',a\bigr)\bigr)\mathbf{p}_{n}(a,y).
\]
Let us split the sum on $a\in G$ into two sums $I+\mathit{II}$, where $I$ is
the sum over $a\in\Bb_x(\Dd(x,y)/2)$, and $\mathit{II}$ on the remaining
$a$. Using Cauchy--Schwarz we can write
\[
I^2\le \biggl(\sum_{a\in
\Bb_x(\Dd(x,y)/2)}\bigl(
\mathbf{p}_{n}(x,a)-\mathbf{p}_{n-1}\bigl(x',a
\bigr)\bigr)^2 \biggr) \biggl(\sum_{a\in\Bb_x(\Dd(x,y)/2)}
\mathbf{p}_{n}(a,y)^2 \biggr).
\]
For the first term, bound the denominator in the definition of
$\Delta_n$ by its maximum and get
\begin{eqnarray*}
&&\sum_{a\in
\Bb_x(\Dd(x,y)/2)}\bigl(\mathbf{p}_{n}(x,a)-
\mathbf{p}_{n-1}\bigl(x',a\bigr)\bigr)^2\\
&&\qquad\le
\Delta_n\bigl(x,x'\bigr)^2 \cdot\max
_{a\in\Bb_x(\Dd(x,y)/2) }\bigl\{\mathbf{p}_{n}(x,a) +
\mathbf{p}_{n-1}\bigl(x',a\bigr)\bigr\}.
\end{eqnarray*}
For the second term write
\begin{eqnarray*}
\sum_{a\in\Bb_x(\Dd(x,y)/2)}\mathbf{p}_{n}(a,y)^2&
\le& \Bigl(\max_{a\in\Bb_x(\Dd(x,y)/2)}\mathbf{p}_{n}(a,y) \Bigr)\cdot
\biggl(\sum_{a\in\Bb
_x(\Dd(x,y)/2)}\mathbf{p}_{n}(a,y)
\biggr)
\\
&\le& \Bigl(\max_{a\in\Bb_x(\Dd(x,y)/2)}\mathbf{p}_{n}(a,y) \Bigr)
\cdot \biggl(\sum_{a\in\Bb_x(\Dd(x,y)/2)}d\cdot\mathbf{p}_{n}(y,a)
\biggr)
\\
&\le& d\cdot \Bigl(\max_{a\in
\Bb_x(\Dd(x,y)/2)}\mathbf{p}_{n}(a,y)
\Bigr).
\end{eqnarray*}
Together we get
\begin{eqnarray}\label{qq}
I^2&\le& d\cdot\Delta_n\bigl(x,x'
\bigr)^2 \cdot\max_{a\in\Bb_x(\Dd(x,y)/2)
}\bigl\{\mathbf{p}
_{n}(x,a) +\mathbf{p}_{n-1}\bigl(x',a\bigr)\bigr
\}
\nonumber
\\[-8pt]
\\[-8pt]
\nonumber
&&{}\times \max_{a\in
\Bb_x(\Dd(x,y)/2)}\mathbf{p}_{n}(a,y).
\end{eqnarray}
Now, the second maximum in the right-hand side of \eqref{qq} is a
maximum on a smaller set than the first maximum in \eqref{zz} [note
that points in $\Bb_x(\Dd(x,y)/2)$ are at distance larger than
$\Dd(x,y)/2$ from $y$]. Similarly, the first maximum is smaller than $(1+d)$
times the second maximum of \eqref{zz}. Therefore, the product of
maxima is smaller than
\[
(d+1)\cdot\mathop{\max_{a,b\in
\Bb_x(2n)\dvtx }}_{\Dd(a,b) \geq\Dd(x,y)/2}
\mathbf{p}_{n}(a,b)\cdot\max_{a,b\in
\Bb_x(2n)}
\mathbf{p}_n(a,b).
\]
The estimate for $\mathit{II}$ is similar:
\[
%&\left(\sum_{a\notin\Bb_y(\Dd(x,y)/2)}(\pp_{n}(y,a)-\pp_{n-1}(z,a))^2
%\right)\left(\sum_{a\notin\Bb_y(\Dd(x,y)/2)}\pp_{n}(a,x)^2\right)
%\nonumber\\
%&\le~\left(\sum_{y\in G}\frac{(\pp_{n}(\rho,y)-\pp_{n-1}(\tilde
%X_1,y))^2}{\pp_{n}(\rho,y)}\right)
%\big(\max_{y\notin\Bb_\rho(\Dd(x,\rho)/2)}\pp_{n}(\rho,y)\big)\big(
%\max_{y\in G}\pp_{n}(y,x)\big)
\mathit{II}^2
\leq d\cdot\Delta_n\bigl(x,x'\bigr)^2 \cdot
\max_{a\notin\Bb_x(\Dd(x,y)/2)
}\bigl\{\mathbf{p}_{n}(x,a) +
\mathbf{p}_{n-1}\bigl(x',a\bigr)\bigr\} \cdot\max
_{a\notin\Bb_x(\Dd(x,y)/2)}\mathbf{p}_{n}(a,y).
\]
It is easy to obtain the same bound again, except the estimates are
reversed ({i.e.}, what was bounded by the first term before is now
bounded by the second term). We sum up:
\begin{eqnarray*}
\bigl(\mathbf{p}_{2n}(x,y)-\mathbf{p}_{2n-1}
\bigl(x',y\bigr) \bigr)^2 &=& (I+\mathit{II})^2\le 2
\bigl(I^2+\mathit{II}^2\bigr)
\\
&\le&4d(d+1)\cdot\Delta_n\bigl(x,x'\bigr)^2
\cdot\mathop{\max_{a,b\in\Bb
_x(2n)\dvtx }}_{\Dd(a,b) \geq
\Dd(x,y)/2}\mathbf{p}_{n}(a,b)\\
&&{}\times\max_{a,b\in\Bb_x(2n)}\mathbf{p}_n(a,b). %\qedhere
\end{eqnarray*}
\upqed\end{pf}

In Section~\ref{sec:entropy} it was always enough to discuss behavior
(say of $H_n-H_{n-1}$) on a sequence $n_k$. Here it is no longer enough
and we need an estimate that holds for all $n$. Hence we prove:

%A final tool for the proof of Theorem~\ref{dimension linear} is a
%slight strengthening of Theorem~\ref{no sublinear harmonic} for the
%case of percolation. Theorem~\ref{no sublinear harmonic} required the
%function to be sublinear at every scale, but here it will be important
%to relax this assumption to work at only an (arbitrarily sparse)
%sequence of scales.

\begin{lemma}\label{lem:alln}For supercritical percolation,
$H_n-H_{n-1}\le C/n$ for
every $n$, where $C$ is a constant depending only on $d$ and $p$.
\end{lemma}

%One may compare this lemma to Section~\ref{sec:entropy} where the
%inequality $H_n-H_{n-1}\le C/n$ was shown to hold for a
%subsequence of $n$. Here, we need the claim for all $n$.

%
\begin{pf} The heat kernel estimates \eqref{a} show, after a little
calculation, that
\begin{equation}
\label{yy} \Hh_n=\frac{d}2\log n+O(1) \qquad\forall
n>n_0(\omega).
\end{equation}
For $n\le n_0(\omega)$ we can use a much rougher bound, say $\Hh_n\le
d\log(2n)$ which follows from the fact that for any cluster $\omega$
the distribution of $R_n$ is supported on the cube $\{-n,\ldots,n\}^d$
and any measure has entropy smaller than the entropy of the uniform
measure on its support. Since $n_0(\omega)$ has a stretched
exponential tail, we can integrate over the environment and get that
$H_n=\frac{d}2\log n+O(1)$. This means that $H_{n}-H_{n/2}\le C$ for some $C$. Using
the fact that $H_n-H_{n-1}$ is decreasing (Corollary~\ref{cor:decreasing} on page \pageref{cor:decreasing}) proves the lemma.
\end{pf}

\begin{pf*}{Proof of Theorem~\ref{heat kernel}}
As before, percolation can be seen as a stationary random graph, and it
is sufficient to prove
\[
\mathbb E \bigl(\bigl(\mathbf{p}_{2n}(\rho,x)-\mathbf{p}_{2n-1}(X_1,x)
\bigr)^2\cdot {\mathbf 1}_{\{x\in\omega\}} \bigr)\le\frac{C'_3}{n^{d+1}}
\exp\bigl(-C'_4\Dd (x,\rho)^2/n\bigr),
\]
where $C'_3$ and $C'_4$ depend only on $d$ and the percolation
probability $p$.

Again we use the variables $n_y(\omega)$ from
\eqref{a} and \eqref{eq:stretched}. Take $\varepsilon$ to be given
by the
stretched exponential bound \eqref{eq:stretched} for
$n_y(\omega)$. Note that we can restrict
ourselves to $|x|\le n^{1/2+\varepsilon/3}$, since in the regime
$|x|\ge
n^{1/2+\varepsilon/3}$, the heat kernel decreases fast enough so that one
can tune the constant $C'_4$ in order to obtain the result
for free. Fix therefore $|x|\le n^{1/2+\varepsilon/3}$.
Let $N(\omega) = \max\{ n_y(\omega) \dvtx y \in\Bb_{\rho}(n) \}$.
The Gaussian estimates \eqref{a} imply that for a.e. environment
$\omega$ such that $x\in\omega$,
whenever $n\geq N(\omega)$, we have
\begin{eqnarray}
\label{xy} \mathop{\max_{a,b\in\Bb_\rho(2n)\dvtx }}_{ \Dd(a,b)>\Dd(\rho,x)/2}
\mathbf{p}
_n(a,b)& \leq&\frac{C_3}{n^{d/2}}\exp \bigl[-C_4\Dd(x,
\rho)^2/n \bigr] \quad\mbox{and}
\nonumber
\\[-8pt]
\\[-8pt]
\nonumber
 \max_{a,b\in\Bb_\rho(2n)}
\mathbf{p}_n(a,b)& \leq&\frac{C_3}{n^{d/2}}.
\end{eqnarray}
%
%Taking $y=\rho$ and $z=\tilde X_1$ and
Averaging \eqref{zz} on the environments satisfying $N(\omega)\le n$
(for which we have \eqref{xy}), we find
\begin{eqnarray*}
&&\mathbb E \bigl[\bigl(\mathbf{p}_{2n}(\rho,x) -\mathbf{p}_{2n-1}(
X_1,x)\bigr)^2\cdot{\mathbf1}_{\{x\in\omega\}}{
\mathbf1}_{\{N(\omega)
\le n\}} \bigr]
\\
&&\qquad
 \le4d(d+1)\cdot\mathbb E \bigl[\Delta_n(\rho,X_1)^2
\bigr]\cdot \frac
{C_3^2}{n^{d}}\exp \bigl[-C_4\Dd(x,
\rho)^2/n \bigr].
\end{eqnarray*}
%
%Now, note that for every $n>0$, $H(\tilde X_n)-H(\tilde X_{n-1})=
%\mathbb E[\Hh(\tilde X_1|\tilde X_n)]-H(\tilde X_1)$. Since
%$$\Hh(\tilde X_1|\tilde X_n)=~\Hh(\tilde X_1|\tilde X_n,\tilde
%X_{n+1})~\le~\Hh(\tilde X_1|\tilde X_{n+1}),$$
%the sequence $(H(\tilde X_n)-H(\tilde X_{n-1}))_n$ is decreasing.
We now apply Theorem~\ref{entropy} to bound $\E[\Delta_n^2]$ by
$2(H_n-H_{n-1})$. Recall also that Lemma~\ref{lem:alln}
says that for
supercritical percolation $H_n-H_{n-1}\le C/n$ for all $n$. Together
these give
\[
\mathbb E \bigl[\bigl(\mathbf{p}_{2n}(\rho,x)-\mathbf{p}_{2n-1}(
\tilde X_1,x)\bigr)^2\cdot {\mathbf1}_{\{x\in\omega\}}{
\mathbf1}_{\{N(\omega)\le n\}} \bigr] \le\frac{C}{n^{d+1}}\exp \bigl[-C_4
\Dd(x,\rho)^2/n \bigr].
\]
We do not need to control the behavior of the gradient on $\{N(\omega)
> n\}$
since this event has probability at most $C n^d e^{-n^\varepsilon}$.
Hence in the regime $|x|\le n^{1/2+\varepsilon/3}$ we find
\begin{eqnarray*}
&&\mathbb E \bigl[\bigl(\mathbf{p}_{2n}(\rho,x)-\mathbf{p}_{2n-1}(
\tilde X_1,x)\bigr)^2\cdot {\mathbf1}_{\{x\in\omega\}}{
\mathbf1}_{\{N(\omega) > n\}} \bigr]
\\
&&\qquad\le\mathbb P_p\bigl(N(\omega)> n\bigr)\le
\frac
{C}{n^{d+1}}\exp \bigl[-C_4\Dd(x,\rho)^2/n \bigr].
\end{eqnarray*}
Putting all the pieces together, we obtain the result.
\end{pf*}

The proof involved only Gaussian estimates at mesoscopic scale and the
entropy argument. It extends to other contexts such as random
conductances satisfying the uniform elliptic condition (see Example~\ref{random conductance}). One may then get, using convolution,
annealed second space-derivative and first time-derivative estimates
for the heat kernel using the first space-derivative estimates. We
refer to Section~5 of \cite{DD05} for more details.
\section{Open questions}\label{sec:open question}

This article must be understood as an introduction and some initial
steps in the subject.
There are many natural questions on harmonic functions which remain
open. We present few of them in this section.

%pa6.subsection.subsubsection.1 #&#
\subsection*{Minimal growth harmonic functions}
The question of minimal growth harmonic functions was implicitly
studied in the literature: the failure of the Liouville property
corresponds to a special case of minimal growth. When the Liouville
property is true, it becomes interesting to determine the minimal
growth. Even the deterministic case (i.e., transitive or Cayley graphs)
has interesting phenomenology, and we plan to analyze some examples in
a future paper. Note that groups always admit linear growth harmonic
functions \cite{Kle10,ST,T10}. This is no longer the case for
stationary random graphs. When the random walk is
subdiffusive (note that the random walk on Cayley graphs is at least
diffusive, a result due to Erschler; see Lee and Peres \cite{LP09}),
Theorem \ref{page:2p} (page \pageref{page:2p}) implies a phenomenon which is
specific to random environments.

%co23 #&#
\begin{corollary}
Let $(G,\nu,\rho)$ be a stationary random graph with polynomial growth
such that the random walk is (strictly) subdiffusive.
Then, almost surely there do not exist linear growth harmonic functions.
\end{corollary}

Therefore graphical fractals, UIPQ, critical Galton--Watson trees
conditioned to survive
and the incipient infinite cluster (IIC) do not admit linear growth
harmonic functions. We mention that it was already proved \cite{BC10}
that the uniform infinite planar triangulation is almost surely
Liouville. There are no nonconstant harmonic functions on the
critical Galton--Watson tree or on the IIC, as both have infinitely many
cut vertices. Indeed, the Galton--Watson tree is well known to be
one-ended and hence, as a tree, must have infinitely many cut
vertices. The existence of cut points for the IIC is essentially
known, but we did not find a reference and including a full proof
would take us too far off-topic.
%and the existence of cut-points for the IIC is trivial using the BK
%inequality \cite{AJ}). %Therefore, a negative answer to the previous
%question could be implied by an affirmative answer to the following
%question.

%qu1 #&#
\begin{question}
Do there exist nonconstant harmonic functions with polynomial growth
on the UIPQ?
\end{question}

If such functions exist, we may ask the following question:

%qu2 #&#
\begin{question}
What is the minimal growth of a nonconstant harmonic function on the UIPQ?
\end{question}

%More generally, it is natural to ask if the minimal growth of harmonic
%functions on a stationary random graph can be strictly larger than
%linear but still polynomial.

%pa6.subsection.subsubsection.2 #&#
\textit{Space of harmonic functions with polynomial growth}.
Cayley graphs with polynomial growth automatically satisfy the volume
doubling property and the Poincar\'e inequality, thus implying that
spaces of harmonic functions with prescribed polynomial growth are
finite dimensional.
The possibility of such behavior in the case of stationary random
graphs of polynomial volume growth is a legitimate question. For example:
%\begin{question}
%Let $(G,\nu,\rho)$ be a stationary random graph with polynomial growth.
%Is the space of harmonic functions with prescribed polynomial growth
%finite dimensional?
%\end{question}
%
%The difficulty comes from the fact that we do not necessarily have
%Poincar\'e inequality at our disposal (in the case of the UIPQ for
%instance). Therefore, we cannot use the technology developed in
%Section~\ref{sec:polynomial growth} in the general context. We mention
%that there exists another strategy to prove finite dimensionality,
%proposed in \cite{Del98}, relying on the following weaker statement:
%for every harmonic function $h$ on $G$ and $x\in G$\dvtx %$$h^2(x) \le \frac{C}{\nu(\Bb_x(n))}\sum_{y\in\Bb(x,Cn)}h^2(y)\nu(y).$$
%where is $C$ a constant independent of $x$ and $n$. This inequality
%appears in standard
%proofs of elliptic Harnack inequalities and holds for a larger class
%than those satisfying the doubling volume property and Poincar\'e
%inequality. Still one cannot expect it to hold in graphs with small
%``bottlenecks'' like the UIPQ. %We do not know if this inequality
%should be expected in the UIPQ context.

%qu3 #&#
\begin{question}
Is the space of harmonic functions with some prescribed polynomial
growth on the UIPQ finite dimensional?
\end{question}

%pa6.subsection.subsubsection.3 #&#
\textit{Dimension of spaces of harmonic functions.}
The computation of the dimension of spaces of harmonic functions does
not restrict to the case of linear growth harmonic functions. For a
graph $G$ and $k>0$, let $d_k[G]$ be the dimension of the space of
harmonic functions with growth bounded by a polynomial of degree $k$.

The similarity between $\Z^d$ and the infinite cluster of percolation
might extend to the dimension of the space of harmonic functions with
arbitrary polynomial growth. More precisely, we ask the following question:

%qu4 #&#
\begin{question}
Are the families $(d_k[\omega])_{k> 0}$ and $(d_k[\Z^d])_{k>0}$ equal
almost surely?
\end{question}

In particular, an interesting intermediate step toward this question
would be to show that there is no harmonic function with noninteger growth.

It is natural to ask if an invariance principle for the random walk in
the random environment $\omega$ implies that the sequence
$(d_k[\omega])$ coincides with $(d_k[\Z^d])$. On $\Z^d$, diffusivity
and the invariance principle are robust under rough
isometry. Therefore, one can ask if $(d_k[G])_{k\geq0}$ is invariant
under rough isometry for these kind of graphs. This is not true in
general. For instance, the Liouville property is not invariant under
rough isometry; see \cite{Lyo87} for the first example or \cite{BS96}
for a simpler one.

More generally, one can ask whether a small perturbation of a Cayley
graph modifies drastically the harmonic functions on it. For instance,
consider percolation on a Cayley graph $G$ such that $p_u(G)$ (the
infimum of the values for which there exists a \emph{unique} infinite
cluster) is strictly smaller than 1. Fix $p>p_u(G)$, and set $\omega
(G)$ to be the unique infinite cluster of the percolation with
parameter $p$.

%qu5 #&#
\begin{question}
Are the dimensions of spaces of harmonic functions with a given growth
equal for $G$ and $\omega(G)$?
\end{question}

Note that the question, in the case of bounded harmonic functions on
the infinite percolation cluster for nonamenable Cayley graphs, was
addressed in \cite{BLS99}.

In the context of Cayley graphs, the space of harmonic functions with a
certain growth rate is crucial in the study of the underlying group.
Indeed, the latter acts on harmonic functions naturally. In the random
setting, we do not have this interpretation. Nevertheless, an
interesting question is to understand what information on the random
graph is encoded in the sequence $(d_k[G])_{k\geq0}$. In particular,
the following question would be a first step in this direction:

%qu6 #&#
\begin{question}
Consider a random subgraph $G$ of $\Z^d$. What are the requirements to
ensure that $(d_k[G])_{k\ge0}$ equals $(d_k[\Z^d])_{k\ge0}$?
\end{question}

%pa6.subsection.subsubsection.4 #&#
\section*{Acknowledgments}We wish to thank Ofer Zeitouni for
pointing out that the entropy argument works also in nonreversible
setting; Russell Lyons for directing
us to \cite{EK10} and Jean-Dominique Deuschel for a discussion on
connectivity leading to the last part of Section~\ref{sec:entropy}.
This paper was written during the visit of the
second author to the Weizmann Institute in Israel. The first author is
the incumbent of the Renee and Jay Weiss Professorial Chair.

%%%%%%%%%%%%%%%%%%%%%%%%%%%%%%%%%%%%%%%%
%
%\newcommand{\nm}[1]{\textsc{#1}}
%\newcommand{\pt}[1]{{#1}}
%\newcommand{\jn}[1]{\emph{#1}}
%
%\bibliographystyle{alpha}
% imsref loaded by akundreckaite, 2014-06-03 13:46:50
% imsref loaded by akundreckaite, 2014-06-03 13:57:39
% imsref loaded by akundreckaite, 2014-06-03 14:54:39
%

\printaddresses
\end{document}